\documentclass{article}
\usepackage{arxiv}

\usepackage[utf8]{inputenc}
\usepackage[T1]{fontenc}
\usepackage{hyperref}
\usepackage{url}
\usepackage{booktabs}
\usepackage{amsfonts}
\usepackage{nicefrac}
\usepackage{microtype}
\usepackage[numbers]{natbib}

\usepackage{graphicx}
\usepackage{multirow}
\usepackage{amsmath,amssymb,amsfonts}
\usepackage{amsthm}
\usepackage{mathrsfs}
\usepackage[title]{appendix}
\usepackage{xcolor}
\usepackage{textcomp}
\usepackage{manyfoot}
\usepackage{booktabs}
\usepackage{algorithm}
\usepackage{algorithmicx}
\usepackage{algpseudocode}
\usepackage{listings}

\usepackage{subcaption}
\usepackage{amsmath,amssymb,amsfonts}
\usepackage{mathtools}
\usepackage{physics}
\usepackage{bm}
\usepackage{siunitx}
\usepackage{pgfplots}
\usepgfplotslibrary{external}
\usepgfplotslibrary{polar}

\tikzsetfigurename{figure}
\pgfplotsset{compat=1.18}
\usetikzlibrary{positioning}

\theoremstyle{plain}
\newtheorem{theorem}{Theorem}

\newtheorem{lemma}[theorem]{Lemma}
\newtheorem{corollary}[theorem]{Corollary}

\theoremstyle{remark}

\theoremstyle{definition}

\DeclareMathOperator{\fl}{fl}

\DeclareMathOperator{\diag}{diag}
\DeclareMathOperator{\offdiag}{offdiag}

\DeclareMathOperator{\dev}{dev}

\DeclareMathOperator{\cof}{cof}
\newcommand{\inner}[2]{\langle #1, #2 \rangle}

\newcommand{\T}{\mathsf{T}}

\newcommand{\epsmach}{\epsilon_\text{mach}}

\newcommand{\ErrorPlot}[2]{
  \centering
  \begin{tikzpicture}
    \begin{loglogaxis}[
        width=7cm,
        height=5cm,
        xlabel={Perturbation parameter $\delta$},
        ylabel={Absolute forward error},
        legend pos=south east,
        tick label style={font=\small},
        label style={font=\small},
        legend style={font=\small, at={(0.5,1.03)}, anchor=south, /tikz/every even column/.append style={column sep=0.5cm}},
        legend columns = 2,
        clip mode=individual
      ]
      \addplot+[red, only marks,mark=triangle] table [x=delta,y=#2_errors_naive] {#1};
      \addlegendentry{naive}

      \addplot+[green, only marks,mark=o] table [x=delta,y=#2_errors_naive_tensor] {#1};
      \addlegendentry{naive tensor}

      \addplot+[blue, only marks,mark=x] table [x=delta,y=#2_errors_c] {#1};
      \addlegendentry{present}

      \addplot[black, very thick] table [x=delta,y=#2_conds] {#1};
      \addlegendentry{stability bound}
    \end{loglogaxis}
  \end{tikzpicture}
}

\newcommand{\ErrorPlotEigvals}[4]{
    \centering
    \begin{tikzpicture}
        \begin{loglogaxis}[
                width=7cm,
                height=5cm,
                xlabel={Perturbation parameter $\delta$},
                ylabel={Absolute forward error},
                legend pos=south east,
                tick label style={font=\small},
                label style={font=\small},
                legend style={font=\small, at={(0.5,1.03)}, anchor=south, /tikz/every even column/.append style={column sep=0.5cm}},
                legend columns = 2,
                clip mode=individual,
                ymin=#3,
                ymax=#4
            ]
                \addplot+[blue,only marks,mark=x] table [x=delta,y=#2_errors_c] {#1};
                \addlegendentry{present}

                \addplot+[red, only marks,mark=triangle] table [x=delta,y=#2_errors_naive] {#1};
                \addlegendentry{naive}

                \addplot+[green, only marks,mark=asterisk] table [x=delta,y=#2_errors_lapack] {#1};
                \addlegendentry{LAPACK DGEEV}

                \addplot[black, thick] table [x=delta,y=#2_conds] {#1};
                \addlegendentry{stability bound}
        \end{loglogaxis}
    \end{tikzpicture}
}

\begin{document}

\title{Numerically stable evaluation of closed-form expressions
    for eigenvalues of $3 \times 3$ matrices}

\author{
    Michal Habera\thanks{These authors contributed equally to this work.} \\
    Department of Engineering \\
    University of Luxembourg \\
    Esch-sur-Alzette, Luxembourg \\
    \texttt{michal.habera@uni.lu} \\
    \And
    Andreas Zilian\footnotemark[1] \\
    Department of Engineering \\
    University of Luxembourg \\
    Esch-sur-Alzette, Luxembourg \\
    \texttt{andreas.zilian@uni.lu} \\
}

\maketitle

\begin{abstract}
    Trigonometric formulas for eigenvalues of $3 \times 3$ matrices that build on Cardano's and
    Viète's work on algebraic solutions of the cubic are numerically unstable for matrices with
    repeated eigenvalues. This work presents numerically stable, closed-form evaluation of eigenvalues
    of real, diagonalizable $3 \times 3$ matrices via four invariants: the trace $I_1$, the deviatoric
    invariants $J_2$ and $J_3$, and the discriminant $\Delta$. We analyze the conditioning of these
    invariants and derive tight forward error bounds. For $J_2$ we propose an algorithm and prove its
    accuracy. We benchmark all invariants and the resulting eigenvalue formulas, relating observed
    forward errors to the derived bounds. In particular, we show that, for the special case of
    matrices with a well-conditioned eigenbasis, the newly proposed algorithms have errors within the
    forward stability bounds. Performance benchmarks show that the proposed algorithm is approximately
    ten times faster than the highly optimized LAPACK library for a challenging test case, while
    maintaining comparable accuracy.
\end{abstract}

\keywords{eigenvalues, 3x3 matrices, numerical stability, matrix invariants, discriminant}

\section{Introduction and motivation}

The classical textbook formulas for closed-form expressions of eigenvalues of a diagonalizable
matrix $\mathbf A \in \mathbb R^{3 \times 3}$ with real spectrum are based on the trace of the matrix $I_1$,
and two deviatoric matrix invariants $J_2$ and $J_3$,
\begin{equation}
    \begin{aligned}
        I_1(\mathbf A) & \coloneqq \tr (\mathbf A),                                                                                            \\
        J_2(\mathbf A) & \coloneqq -\frac12 \left[\tr(\dev(\mathbf A))^2 - \tr(\dev(\mathbf A)^2)\right] = \frac{1}{2} \tr(\dev(\mathbf A)^2), \\
        J_3(\mathbf A) & \coloneqq \det (\dev(\mathbf A)). \label{eq:principal-invariants}
    \end{aligned}
\end{equation}
The three eigenvalues $\lambda_k$ are then given by
(see \citet{Smith1961Eigenvalues} or \citet[\S 1.6.2.3]{Bronshtein2015Handbook} or \citet[Eq. 5.6.12]{Press2007Recipes}),
\begin{equation}
    \lambda_k = \frac{1}{3} \left(I_1 + 2 \sqrt{3 J_2} \cos \left( \frac{\varphi + 2 \pi k}{3} \right) \right), \quad k \in \{1, 2, 3\},
    \label{eq:eigvals-trig}
\end{equation}
where the triple-angle $\varphi$ is computed as
\begin{equation}
    \varphi \coloneqq \arccos \left( \frac{3 \sqrt{3}}{2} \frac{J_3}{J_2^{3/2}} \right).
\end{equation}

The above expressions are notoriously unstable in finite-precision arithmetic, especially when
eigenvalues coalesce. A typical pitfall of closed-form approaches is the reduction of the eigenvalue
problem to the computation of roots of a cubic polynomial, see Fig.
\ref{fig:typical-eigvals-approach}. This approach, i.e., the computation of the roots of a cubic
polynomial given its monomial coefficients, is known to be ill-conditioned, see
\citet[p.~110]{Trefethen1997Numerical} and \citet[\S 26.3.3.]{Higham2002Accuracy}. While we do not
bypass the utilization of the characteristic polynomial, we try to improve the numerical stability
of the overall process by improving the stability of the individual steps, and potentially computing
additional, seemingly redundant invariants that help stabilize the computation of the eigenvalues.

\begin{figure}
    \centering
    \begin{tikzpicture}[node distance=2cm, auto, scale=1.0, transform shape]
        \node (matrix) [draw, rectangle] {$\begin{bmatrix} A_{11} & A_{12} & A_{13} \\ A_{21} & A_{22} & A_{23} \\ A_{31} & A_{32} & A_{33} \end{bmatrix}$};
        
        \node (poly) [draw, rectangle, right=of matrix] {$p(\lambda) = \lambda^3 - I_1 \lambda^2 + I_2 \lambda - I_3$};
        \node (roots) [draw, rectangle, right=of poly] {$\lambda_1, \lambda_2, \lambda_3$};
        
        \draw [->, thick] (matrix) -- (poly) node [midway, above] {\small char. poly.};
        \draw [->, thick] (poly) -- (roots) node [midway, above] {\small roots};
        \draw [->, thick] (poly) -- (roots) node [midway, below] {\small ill-conditioned};
        
        \draw [->, thick, bend left=20] (matrix) to node[midway, below]{\small well-conditioned} (roots);
        \draw [->, thick, bend left=20] (matrix) to node[midway, above]{\small eigenvalues} (roots);
        
    \end{tikzpicture}
    \caption{Typical approach for computing eigenvalues of $3 \times 3$ matrices via characteristic polynomial and its roots.}
    \label{fig:typical-eigvals-approach}
\end{figure}
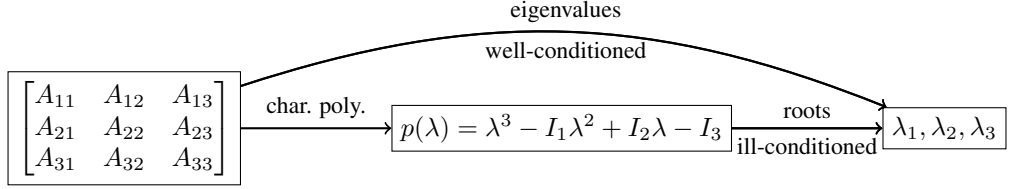

According to \citet{Blinn2007Howto}, the first known approach for improving the numerical stability
is from \citet{LaPorte1973Formulation} who proposed to use the identity $\tan(\arccos x) = \sqrt{1 - x^2} / x$
in the context of solving roots of a general cubic polynomial. When applied to the matrix
eigenvalue problem, the triple-angle expression takes the form
\begin{equation}
    \varphi = \arctan \left( \frac{\sqrt{27 (4 J_2^3 - 27 J_3^2)}}{27 J_3} \right) = \arctan \left( \frac{\sqrt{27 \Delta}}{27 J_3} \right)
    \label{eq:triple-angle-arctan}
\end{equation}
making use of the matrix \emph{discriminant} $\Delta \coloneqq 4 J_2^3 - 27 J_3^2$.
Eq.~\eqref{eq:triple-angle-arctan} has the advantage of evaluating the $\arctan(x) = x - x^3 / 3 + \mathcal{O}(x^5)$
around zero (for matrices with repeated eigenvalues), which is numerically more stable than evaluating the $\arccos(x)$
around one.

Another notable improvement is based on the work of \citet{Scherzinger2008Robust}, who proposed an
algorithm for \textit{symmetric} $3 \times 3$ matrices based on computing the distinct eigenvalue first, then
deflating the matrix to a $2 \times 2$ problem for which Wilkinson's shift is used to compute the
remaining eigenvalues. This approach is stable for symmetric matrices, but it does not
generalize to nonsymmetric matrices. In addition, it is not a closed-form expression and requires
branching and conditional statements.

The computation of the matrix discriminant $\Delta$ itself is also prone to numerical
instability, as it involves subtraction of two potentially close quantities, $4 J_2^3$
and $27 J_3^2$. The first work addressing this issue in the context of $3 \times 3$ matrices is
\citet{Habera2021Symbolic} and is based on the factorization of the discriminant from \citet{Parlett2002Discriminant}
into a sum of products of terms that vanish as the matrix approaches a matrix with multiple eigenvalues.

Recently, an alternative factorization of the discriminant $\Delta$ for symmetric matrices based on
the Cayley--Hamilton theorem was proposed in \citet{Harari2022Computation}. The authors then published
a follow-up paper \citep{Harari2023Using} where the Cayley--Hamilton factorization is abandoned in favor of
a simpler \emph{sum-of-squares} formula for the discriminant.

In \citet{Habera2021Symbolic} we advocated replacing the traditional discriminant
expressions with \emph{sum-of-products} or \emph{sum-of-squares} formulas that avoid catastrophic
cancellation. Unfortunately, as discussed in \citet{Habera2021Symbolic}, the proposed algorithm
failed to achieve eigenvalues with satisfactory accuracy for matrices with $J_2 \to 0$.
In addition, the benchmarks and interpretation of errors were intuitive, but lacked rigorous
forward or backward error analysis. We used scaled invariants $\Delta_p = 3 J_2$ and
$\Delta_q = 27 J_3$. In the present work, we use the classical definitions of the invariants
for consistency with the existing literature, especially in the engineering community where
$J_2$ and $J_3$ are widely used in constitutive modeling of materials. In addition, we
improve the numerical stability in the limit case $J_2 \to 0$ by proposing improved algorithms
for the computation of $J_2$, $J_3$, and $\Delta$.

The lack of error analysis is a common issue in the existing literature on closed-form expressions
for eigenvalues of $3 \times 3$ matrices. Terms like ``numerically stable'' or ``robust'' are often
used without rigorous justification or derivation of error bounds. We address this gap.
Additionally, only in \citet{Habera2021Symbolic} and this work is the numerical stability for the
\textit{generalized case of nonsymmetric matrices} considered.

On the other hand, the typical approach to computing eigenvalues of general matrices uses iterative
algorithms, such as the QR algorithm, which are implemented in standard libraries like LAPACK
\citep{Anderson1999Lapack}. These algorithms are based on numerically stable orthogonal
transformations to reduce the matrix to a simpler form (e.g., Hessenberg form) and then iteratively
applying the QR algorithm to converge to the eigenvalues. Unsurprisingly, these iterative algorithms
are routinely used in practice even for small $3 \times 3$ matrices.

Despite the widespread use of iterative algorithms, closed-form expressions for eigenvalues remain
important due to several reasons: 1. number of floating-point operations is significantly lower than
for iterative algorithms, which is critical in performance-sensitive applications, and 2. they allow
for symbolic differentiation, which is important when sensitivities or gradients are required, e.g.,
in optimization or machine learning applications. The latter was explored in
\citet{Habera2021Symbolic}, where the relation
\begin{equation}
    \mathbf E^\T_k = \frac{\partial \lambda_k}{\partial \mathbf A}
\end{equation}
was used to compute eigenprojectors (i.e., matrices projecting onto the eigenspaces associated with
the eigenvalues $\lambda_k$). With the eigenprojectors available in closed-form, one can compute
functions of matrices (e.g., the matrix exponential) and their derivatives in closed-form as well.
In addition, in the case of a matrix parametrized by some variable $t \in \mathbb R$,
i.e., $\mathbf A(t): t \mapsto \mathbf A(t)$ one can use the closed-form expressions and their
derivatives to study the analytical dependence of eigenvalues and eigenvectors on the parameter $t$.
Lastly, the use of trigonometric solution guarantees that the eigenvalues are ordered
$\lambda_1 \leq \lambda_2 \leq \lambda_3$, which is not the case for iterative algorithms.
Ordering of eigenvalues is important in many applications, e.g., in engineering mechanics when
computing principal stresses or strains.

\section{Numerical stability}
In this work, we use the notation and definitions from \citet{Higham2002Accuracy} and
\citet{Trefethen1997Numerical}. We follow the standard IEEE 754 model with
\begin{equation}
    \fl(x \,\text{op}\, y) = (x\,\text{op}\, y)\,(1 + \delta), \quad \text{op} \in \{+, -, *, /\}.
\end{equation}
The same applies to the floating-point representation of a number, $\fl(x) = x (1 + \delta)$. The
quantity $\delta$ is close to zero. More precisely, it is bounded as $|\delta| \le \epsmach$, where
$\epsmach$ is the unit roundoff (machine precision). In other words, each floating-point operation
of type $(+, -, *, /)$ adds a relative error of at most $\epsmach$. For IEEE 754 double precision,
we have
\begin{equation}
    \epsmach = \tfrac{1}{2}\,\beta^{1-t} = 2^{-53} \approx 1.11 \times 10^{-16},
\end{equation}
where $\beta$ is the base and $t$ is the precision (number of base-$\beta$ digits).

We also use the symbol $\theta_n$ to denote the cumulative relative error of a sequence of $n$ floating-point operations, i.e.,
\begin{equation}
    1 + \theta_n = \prod_{i=1}^n (1 \pm \delta_i)^{\pm 1}, \quad |\delta_i| \le \epsmach,
\end{equation}
with the standard bound (assuming $n\epsmach < 1$)
\begin{equation}
    |\theta_n| \le \frac{n \epsmach}{1 - n \epsmach} = \gamma_n.
\end{equation}

An algorithm $f: V \rightarrow W$ is called \emph{backward stable in the relative sense} if for all
$x \in V$ there exists $\delta x \in V$ such that
\begin{equation}
    \fl(f(x)) = f(x + \delta x), \quad \text{where} \quad \frac{\norm{\delta x}}{\norm{x}} \leq C \epsmach.
\end{equation}
In this work, $V$ and $W$ are finite-dimensional vector spaces. Most often, $V = \mathbb R^{3 \times 3}$
and $W = \mathbb R$. Since we are concerned with small matrices of fixed size $3 \times 3$, the
dependence of the constant $C$ on the problem dimension is negligible. In addition, the constant $C$
is required to be moderate, usually $C \leq 100$, often $C \leq 10$. The symbol $C$ will be used to
denote this constant in the rest of the paper.

The quantity $\norm{\delta x} / \norm{x}$ is called the (relative) \emph{backward error} of the
algorithm. In other words, the algorithm is backward stable if it computes the exact result for a
slightly perturbed input, where the perturbation is small relative to the input.

The \emph{relative condition number} of a function $f: V \rightarrow W$ at $x$ is defined as
\begin{equation}
    \kappa_f(x) \coloneqq \sup_{\delta x} \left( \frac{\norm{f(x + \delta x) - f(x)}}{\norm{f(x)}}
    \bigg/ \frac{\norm{\delta x}}{\norm{x}} \right),
\end{equation}
i.e., the worst-case relative change in the output divided by a relative change in the input. Here,
$\delta x$ is infinitesimal. That is, the above is understood in the limit $\norm{\delta x} \to 0$.
For differentiable functions, the relative condition number can be expressed in terms of the Jacobian
$\mathbf J_f(x) = \partial f / \partial x$ as \citep[Eq. 12.6]{Trefethen1997Numerical},
\begin{equation}
    \kappa_f(x) = \norm{\mathbf J_f(x)} \frac{\norm{x}}{\norm{f(x)}}.
    \label{eq:def-rel-condition-number}
\end{equation}

The \emph{absolute condition number} is defined as
\begin{equation}
    \kappa_f^\text{abs}(x) \coloneqq \sup_{\delta x} \left( \frac{\norm{f(x + \delta x) - f(x)}}{\norm{\delta x}} \right),
\end{equation}
which, using the Jacobian, can be expressed as \citep[Eq. 12.3]{Trefethen1997Numerical},
\begin{equation}
    \kappa_f^\text{abs}(x) = \norm{\mathbf J_f(x)}.
\end{equation}

The error of the floating-point evaluation of an algorithm $f$, $\norm{\fl(f(x)) - f(x)}$ at
a point $x$, is called the \emph{absolute forward error}. We say that an algorithm is
\emph{forward stable in the absolute sense} if its absolute forward error is on the order of
$\kappa_f^\text{abs}$ times the machine precision. An important result used throughout the
paper is that the forward error is bounded by the product of the condition number and the backward error
\begin{equation}
    \text{forward error} \;\le\; \text{condition number} \cdot \text{backward error}.
\end{equation}
The meaning of each term must be consistent: we bound absolute forward error by
absolute condition number and absolute backward error, or relative forward error by relative
condition number and relative backward error. Which of the two is used depends on the context and
the problem at hand.

An algorithm is called \emph{accurate} if it produces results with a small relative forward error,
see \citet[Eq. 14.2]{Trefethen1997Numerical},
\begin{equation}
    \frac{\norm{\fl(f(x)) - f(x)}}{\norm{f(x)}} \leq C \,\epsmach + \mathcal O(\epsmach^2).
\end{equation}
Accurate algorithms produce results that are as close to the exact result as the floating-point
format and machine precision allow and are the pinnacle of what one can achieve in numerical
computations.

\section{Benchmarks}

In this section, the methodology for generating numerical benchmarks is described. It could be the
case that rounding error tests are sensitive to the specific libraries, compilers, and hardware
used. We describe the procedures in detail to allow reproducibility. We also provide the data and
code used to generate the results in this paper as part of open-source library \texttt{eig3x3},
see \citet{Habera2025Eig3x3}.

Algorithms for evaluating the invariants in IEEE 754 double-precision floating-point were
implemented in C11 with Python wrappers via CFFI \citep{Rigo2025CFFI} using the \texttt{double} 64-bit
floating-point format and in NumPy 2.3.4 \citep{Harris2020Array} using the \texttt{numpy.float64} data
type. In order to compute the forward error of a function $f(x)$, we compute the reference value
$f_\text{ref}(x)$ using the mpmath 1.3.0 library \citep{Mpmath2023Python}, with precision
set to a high number of decimal places, i.e., \texttt{mpmath.dps = 256}.

In order to capture several limit cases of the eigenvalue multiplicities and conditioning of the
eigenvectors, we consider test input matrices computed as
\begin{equation}
    \fl(\mathbf A) = \fl(\mathbf U \mathbf D \mathbf U^{-1})
    \label{eq:test-matrix}
\end{equation}
where $\mathbf D = \diag(\lambda_1, \lambda_2, \lambda_3)$ is a diagonal matrix with prescribed
eigenvalues, and $\mathbf U$ is a nonsingular transformation matrix. We evaluate the matrix $\fl(\mathbf A)$
from Eq.~\eqref{eq:test-matrix} using \texttt{numpy.linalg.inv} to compute $\mathbf U^{-1}$ and
\texttt{numpy.matmul} to compute the matrix–-matrix products, all in double precision. The resulting
matrix $\fl(\mathbf A)$ is then used as input to the invariant evaluation algorithms. The floating-point
matrix $\fl(\mathbf A)$ is not guaranteed to have the exact eigenvalues $\lambda_1, \lambda_2, \lambda_3$
of the diagonal matrix $\mathbf D$. Nevertheless, we compute the forward error of an algorithm $f$ as
\begin{equation}
    \text{forward error} = \abs{\fl(f(\fl(\mathbf A))) - f_\text{ref}(\fl(\mathbf A))}.
\end{equation}
An important detail is that we compute the high-precision reference value
$f_\text{ref}(\fl(\mathbf A))$ at the floating-point matrix $\fl(\mathbf A)$, not at the exact matrix
$\mathbf A$.

\begin{figure}[htbp]
    \centering
    \begin{subfigure}[t]{0.49\linewidth}
        \centering
        \begin{tikzpicture}
            \begin{axis}[
                    width=8cm,
                    height=7cm,
                    view={0}{90},
                    samples=100,
                    samples y=100,
                    domain=-1.5:1.5,
                    domain y=0:3,
                    xlabel={Invariant $J_3$},
                    ylabel={Invariant $J_2$},
                    tick label style={font=\small},
                    label style={font=\small},
                    legend style={font=\small}]
                \addplot3 [
                    contour lua={levels={0,10}},
                    thick,
                ] {4*y^3-27*x^2};
                \addlegendentry{$\Delta$}
                \addplot+[
                    only marks,
                ] table[x=J3_values_c, y=J2_values_c] {results/invariants-double_lim_J3J2-u1.dat};
                \addlegendentry{$\delta$}
            \end{axis}
        \end{tikzpicture}
        \caption{Discriminant contour lines with the benchmark path for $\mathbf D_1 = \diag(1, 1, 1+\delta)$. This represents a limiting case
            of $J_2 \to 0$ in which each generated matrix has $\Delta = 0$, meaning we move along the double-eigenvalue path towards
            the triple-eigenvalue.}
        \label{fig:J2-J3-benchmark-d2}
    \end{subfigure}
    \hfill
    \begin{subfigure}[t]{0.49\linewidth}
        \centering
        \begin{tikzpicture}
            \begin{axis}[
                    width=8cm,
                    height=7cm,
                    view={0}{90},
                    samples=100,
                    samples y=100,
                    domain=-1.5:1.5,
                    domain y=0:3,
                    xlabel={Invariant $J_3$},
                    ylabel={Invariant $J_2$},
                    tick label style={font=\small},
                    label style={font=\small},
                    legend style={font=\small}]
                \addplot3 [
                    contour lua={levels={0,10}},
                    thick,
                ] {4*y^3-27*x^2};
                \addlegendentry{$\Delta$}
                \addplot+[
                    only marks,
                ] table[x=J3_values_c, y=J2_values_c] {results/invariants-single_lim_disc_t-u1.dat};
                \addlegendentry{$\delta$}
            \end{axis}
        \end{tikzpicture}
        \caption{Discriminant contour lines with the benchmark path for $\mathbf D_2 = \diag(-1, 1, 1+\delta)$. This represents a limiting case
            of $\Delta \to 0$, but both $J_3$ and $J_2$ stay finite and away from zero, so we move towards a double-eigenvalue
            configuration.}
        \label{fig:J2-J3-benchmark-d1}
    \end{subfigure}
    \caption{Benchmark cases in this paper. The red squares represent the limiting path $\delta \to 0$.}
    \label{fig:J2-J3-benchmarks}
\end{figure}
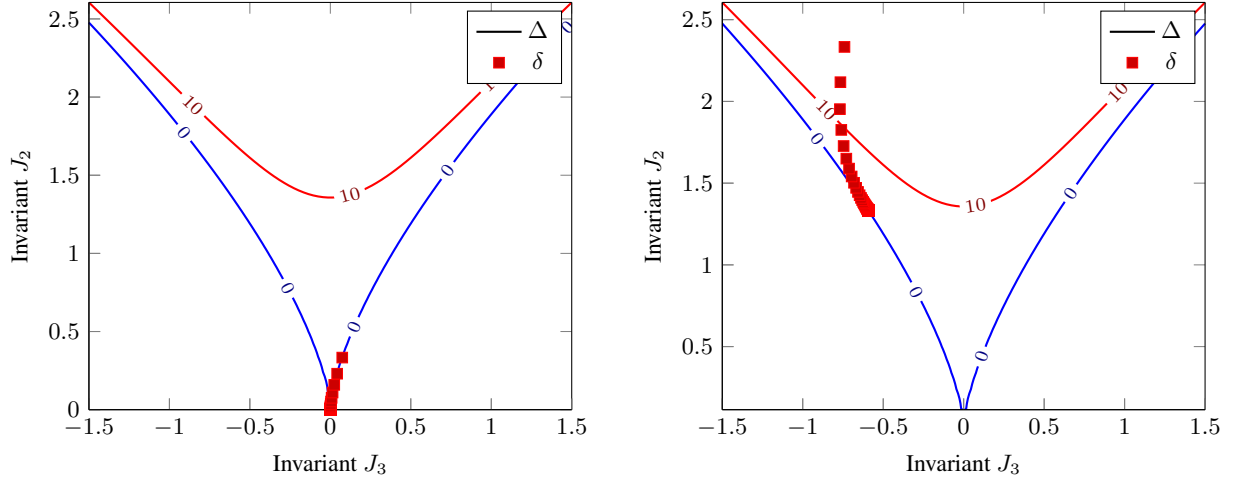

In order to capture the limit cases of eigenvalue multiplicities, we consider two benchmark paths in
this paper, parametrized by a small parameter $\delta \to 0$. The paths are given by
\begin{itemize}
    \item $\mathbf D_1 = \diag(\lambda_1, \lambda_2, \lambda_3) = \diag(1, 1, 1 + \delta),$
          which represents a limiting case of $J_2 \to 0$ and $J_3 \to 0$, moving along the
          double-eigenvalue path towards the triple-eigenvalue,
    \item $ \mathbf D_2 = \diag(\lambda_1, \lambda_2, \lambda_3) = \diag(-1, 1, 1 + \delta),$
          which represents a limiting case of $\Delta \to 0$, but both $J_3$ and $J_2$ stay finite and away
          from zero, so we move towards a double-eigenvalue configuration.
\end{itemize}
The two benchmark paths are illustrated in the $J_3$-$J_2$ plane in Fig.~\ref{fig:J2-J3-benchmarks}.

Transformation matrices $\mathbf U$ used in the benchmarks are chosen as
\begin{equation}
    \begin{aligned}
        \mathbf U_\text{symm} = \begin{bmatrix}
                                    \frac{1}{\sqrt{2}} & -\frac{1}{2}       & \frac{1}{2}        \\
                                    \frac{1}{\sqrt{2}} & \frac{1}{2}        & -\frac{1}{2}       \\
                                    0                  & \frac{1}{\sqrt{2}} & \frac{1}{\sqrt{2}}
                                \end{bmatrix}, \quad
        \mathbf U_1 = \begin{bmatrix}
                          1  & -1 & 1  \\
                          1  & 1  & 1  \\
                          -1 & -1 & 1
                      \end{bmatrix}, \quad
        \mathbf U_2(\gamma) = \begin{bmatrix}
                                  1 & 1 & 1         \\
                                  1 & 0 & 1         \\
                                  2 & 1 & 2+\gamma
                              \end{bmatrix}.
    \end{aligned}
    \label{eq:transformation-matrices}
\end{equation}

Transformation matrix $\mathbf U_\text{symm}$ represents an orthogonal transformation, so the
2-norm condition number is $\kappa_2(\mathbf U_\text{symm}) = 1$ and, as a consequence, any matrix of
the form \eqref{eq:test-matrix} is symmetric.

The matrix $\mathbf U_1$ represents a nonorthogonal transformation matrix with small 2-norm
condition number $\kappa_2(\mathbf U_1) = 2$. Matrices of the form \eqref{eq:test-matrix} with $\mathbf
    U = \mathbf U_1$ are nonsymmetric but have a well-conditioned eigenbasis.

The third case of matrix $\mathbf U_2(\gamma)$ represents a nonorthogonal transformation matrix
with tunable condition number $\kappa_2(\mathbf U_2)$. One can show that $\kappa_2(\mathbf U_2(\gamma)) \to
    \infty$ as $\gamma \to 0$ (as the rows become linearly dependent). Matrices of the form
\eqref{eq:test-matrix} with $\mathbf U = \mathbf U_2(\gamma)$ are nonsymmetric and can have an
arbitrarily ill-conditioned eigenbasis. They represent the most challenging case for numerical
evaluation of invariants and eigenvalues.

\section{Invariant $I_1$}

The first invariant $I_1$ is defined as the trace of the matrix,
\begin{equation}
    I_1(\mathbf A) \coloneqq \tr(\mathbf A) = A_{00} + A_{11} + A_{22}.
\end{equation}

The algorithm for evaluating $I_1$ sums the diagonal elements, as shown in Algorithm~\ref{alg:I1-stable}.

\begin{minipage}{1.0\linewidth}
    \begin{algorithm}[H]
        \begin{algorithmic}
            \Require $\mathbf A \in \mathbb{R}^{3 \times 3}$
            \State $I_1 = A_{00} + A_{11} + A_{22}$
            \State \Return $I_1$
        \end{algorithmic}
        \caption{Evaluation of the invariant $I_1$}
        \label{alg:I1-stable}
    \end{algorithm} \hfill
\end{minipage}

Algorithm~\ref{alg:I1-stable} is trivially backward stable, as the sum of three floating-point
numbers can be seen as the exact sum of slightly perturbed inputs. The floating-point evaluation reads
\begin{equation}
    \begin{aligned}
        \fl(I_1) & = ((A_{00} + A_{11})(1 + \delta_0) + A_{22})(1 + \delta_1)                                       \\
                 & = A_{00}(1 + \delta_0)(1 + \delta_1) + A_{11}(1 + \delta_0)(1 + \delta_1) + A_{22}(1 + \delta_1) \\
                 & = A_{00}(1 + \theta_2) + A_{11}(1 + \theta_2) + A_{22}(1 + \theta_1).
    \end{aligned}
\end{equation}
This is equivalent to a diagonal perturbation of the input matrix $\mathbf A$ with
\begin{equation}
    \begin{aligned}
        \fl(I_1) = I_1(\mathbf A + \delta \mathbf A), \quad \text{where} \quad
        \delta \mathbf A = \begin{bmatrix}
                               A_{00} \theta_2 & 0               & 0                \\
                               0               & A_{11} \theta_2 & 0                \\
                               0               & 0               & A_{22} \theta_1
                           \end{bmatrix}.
    \end{aligned}
\end{equation}
The perturbation is componentwise relatively small, i.e.,
\begin{equation}
    \frac{\abs{\delta A_{ij}}}{\abs{A_{ij}}} \le C \epsmach,
\end{equation}
with $C = 3$ for all $i, j \in \{0, 1, 2\}$.

Since the directional derivative of the trace is
\begin{equation}
    \frac{\partial}{\partial \mathbf A} \tr(\mathbf A)[\delta \mathbf A] = \tr(\delta \mathbf A) = \inner{\mathbf I}{\delta \mathbf A},
\end{equation}
we have that the Jacobian is $\mathbf J_{I_1}= \mathbf I$. This implies the following result:
Algorithm~\ref{alg:I1-stable} for evaluating $I_1$ is forward stable in the sense that the absolute forward error satisfies
\begin{equation}
    \abs{\fl(I_1) - I_1} \leq C \norm{\mathbf A} \epsmach + \mathcal O(\epsmach^2),
\end{equation}
where $C$ is a moderate constant. The relative condition number of $I_1$ is unbounded, as
\begin{equation}
    \kappa_{I_1}(\mathbf A) = \norm{\mathbf J_{I_1}} \frac{\norm{\mathbf A}}{\abs{I_1}} = \frac{C \norm{\mathbf A}}{\abs{\tr(\mathbf A)}}
\end{equation}
where constant $C$ depends on the chosen matrix norm. Thus, we cannot expect any algorithm to be accurate when $\abs{\tr(\mathbf A)}$ is small compared to $\norm{\mathbf A}$.

\section{Invariant $J_2$}
\label{sec:J2-stability}

The invariant $J_2$ (see Eq.~\eqref{eq:principal-invariants}) received a lot of attention
in the literature due to its importance in engineering mechanics and constitutive modeling of materials.
The expression with the use of diagonal differences was in the context of numerical accuracy proposed in
\citet[\S 3.1.2 Eq. 32]{Habera2021Symbolic} and in \citet[\S 4.2.1 Eq. 15]{Harari2023Using}.
In the present work, we consider a generalization of these results to nonsymmetric matrices and provide
rigorous error analysis.

\begin{lemma}{}
    \label{lem:J2-condition}
    The Jacobian of the $J_2$ invariant is
    \begin{equation}
        \mathbf J_{J_2}(\mathbf A) = \dev(\mathbf{A})^\T.
        \label{eq:J2-jacobian}
    \end{equation}
    As a consequence, the $J_2$ invariant is well-conditioned in the absolute sense for
    matrices whose deviatoric part has small norm.
\end{lemma}

\begin{proof}
    We use the following directional derivative
    \begin{equation}
        \frac{\partial}{\partial \mathbf A} \dev(\mathbf{A})[\delta \mathbf{A}] = \dev(\delta \mathbf{A}), \label{eq:dev-derivative}
    \end{equation}
    which follows from the linearity of the deviatoric operator. Combining this with the definition of $J_2$, we have
    \begin{equation}
        \begin{aligned}
            \frac{\partial}{\partial \mathbf A} J_2 [\delta \mathbf{A}] & = \frac{\partial}{\partial \mathbf A} \frac{1}{2} \tr(\dev(\mathbf{A})^2) [\delta \mathbf{A}]                       \\
                                                                        & = \frac{1}{2} \tr \left(\dev(\mathbf{A}) \dev(\delta \mathbf{A}) + \dev(\delta \mathbf{A}) \dev(\mathbf{A}) \right) \\
                                                                        & = \tr \left(\dev(\mathbf{A}) \dev(\delta \mathbf{A}) \right)                                                        \\
                                                                        & = \inner{\dev(\mathbf{A})^\T}{\dev(\delta \mathbf{A})} = \inner{\dev(\mathbf{A})^\T}{\delta \mathbf{A}}.
        \end{aligned}
    \end{equation}
\end{proof}

\begin{minipage}{1.0\linewidth}
    \begin{algorithm}[H]
        \begin{algorithmic}
            \Require $\mathbf{A} \in \mathbb{R}^{3 \times 3}$
            \State $d_{0} = A_{00} - A_{11}$, $d_{1} = A_{00} - A_{22}$, $d_{2} = A_{11} - A_{22}$ \Comment{Diagonal differences}
            \State $\text{offdiag} = A_{01} A_{10} + A_{02} A_{20} + A_{12} A_{21}$ \Comment{Off-diagonal products}
            \State $\text{diag} = \frac{1}{6} (d_0^2 + d_1^2 + d_2^2)$ \Comment{Sum of squares of diagonal differences}
            \State $J_2 = \text{diag} + \text{offdiag}$
            \State \Return $J_2$
        \end{algorithmic}
        \caption{Evaluation of the invariant $J_2$}
        \label{alg:J2-stable}
    \end{algorithm}
\end{minipage}

\begin{lemma}{}
    \label{lem:J2-backward}
    Algorithm~\ref{alg:J2-stable} for evaluating $J_2$ is backward stable in the componentwise relative sense.
\end{lemma}

\begin{proof}
    We note that the final expression for $J_2$ is a sum of two terms, where the first one is based on
    off-diagonal products and the second one is a sum of squares of the diagonal differences. Let us
    examine the sum of squares of the diagonal differences first. Diagonal differences are computed as
    \begin{equation}
        \begin{aligned}
            \fl(d_0) = \fl(A_{00} - A_{11}) = (A_{00} - A_{11})(1 + \delta_0), \\
            \fl(d_1) = \fl(A_{00} - A_{22}) = (A_{00} - A_{22})(1 + \delta_1), \\
            \fl(d_2) = \fl(A_{11} - A_{22}) = (A_{11} - A_{22})(1 + \delta_2).
        \end{aligned}
        \label{eq:diagonal-differences-fwd-error}
    \end{equation}
    and, using Higham's $\theta$-notation, we have
    \begin{equation}
        \fl(\text{diag}) =
        \frac{1}{6} \left(d_0^2 (1 + \theta_6) + d_1^2 (1 + \theta'_6) + d_2^2 (1 + \theta_5) \right).
    \end{equation}
    
    The largest relative error here is $(1 + \theta_6)$, since the first diagonal difference $d_0$
    incurs errors from the subtraction itself, squaring, two additions to the other diagonal
    differences, and one division by 6.
    
    Each off-diagonal product produces a single roundoff error, and summing them together with the
    diagonal term yields
    \begin{equation}
        \begin{aligned}
            \fl(J_2) = A_{01} A_{10}(1 + \theta_5) + A_{02} A_{20} (1 + \theta'_5) + A_{12} A_{21} (1 + \theta_4) \\
            + \frac{1}{6} \left(d_0^2 (1 + \theta_7) + d_1^2 (1 + \theta'_7) + d_2^2 (1 + \theta_6) \right).
        \end{aligned}
    \end{equation}
    
    Here, we already recognize the perturbations required for the off-diagonal terms, i.e.,
    \begin{equation}
        \begin{aligned}
            \mathbf{\delta A} =
            \begin{bmatrix}
                A_{00} \alpha & A_{01} \theta_5 & A_{02} \theta'_5 \\
                0             & A_{11} \alpha   & A_{12} \theta_4  \\
                0             & 0               & A_{22} \alpha
            \end{bmatrix}
        \end{aligned}
        \label{eq:delta-A}
    \end{equation}
    while $\alpha$ for the diagonal perturbation is to be determined.
    For the exact computation with the perturbed input, we have
    \begin{equation}
        \begin{aligned}
            J_2(\mathbf{A} + \mathbf{\delta A})
             & = A_{01} A_{10}(1 + \theta_5) + A_{02} A_{20} (1 + \theta'_5) + A_{12} A_{21} (1 + \theta_4) \\
             & \quad + \frac{1}{6} (1 + \alpha)^2 \left( d_0^2 + d_1^2 + d_2^2 \right).
        \end{aligned}
    \end{equation}
    To match the diagonal contributions we need $\alpha$ such that
    \begin{equation}
        \begin{aligned}
            (1 + \alpha)^2 (d_0^2 + d_1^2 + d_2^2) =
            d_0^2 (1 + \theta_7) + d_1^2 (1 + \theta'_7) + d_2^2 (1 + \theta_6)
        \end{aligned}
    \end{equation}
    which is satisfied for
    \begin{equation}
        \alpha = \sqrt{\frac{d_0^2 (1 + \theta_7) + d_1^2 (1 + \theta'_7) + d_2^2 (1 + \theta_6)}{d_0^2 + d_1^2 + d_2^2}} - 1.
    \end{equation}
    A bound on $\alpha$ follows from the first-order Taylor expansion $\sqrt{1 + x} = 1 + x/2 + \mathcal O(x^2)$ for $x \approx 0$:
    \begin{equation}
        \begin{aligned}
            \left|\alpha\right|
            = & \left| \sqrt{1 + \frac{d_0^2 \theta_7 + d_1^2 \theta'_7 + d_2^2 \theta_6}{d_0^2 + d_1^2 + d_2^2}} - 1 \right|                                   \\
            = & \left|1 + \frac{1}{2} \xi + \mathcal O (\xi^2) - 1\right|                                                           & \text{(Taylor expansion)} \\
            = & \frac{1}{2} \abs{\xi} + \abs{\mathcal O (\xi^2)} \leq \frac{1}{2} \gamma_7 + \left| \mathcal O (\gamma_7^2)\right|. & \text{(see below)}
        \end{aligned}
    \end{equation}
    The last inequality follows from
    \begin{equation}
        \begin{aligned}
            \abs{\xi} = \abs{\frac{d_0^2 \theta_7 + d_1^2 \theta'_7 + d_2^2 \theta_6}{d_0^2 + d_1^2 + d_2^2}}
            \leq \abs{\frac{\max (\theta_7, \theta'_7, \theta_6) (d_0^2 + d_1^2 + d_2^2)}{d_0^2 + d_1^2 + d_2^2}}
            \leq \gamma_7
        \end{aligned}
    \end{equation}
    since the squared diagonal differences are nonnegative.
    
    From the way we constructed the perturbation (i.e., relative to the matrix entries) we now have the
    componentwise relative backward error result
    \begin{equation}
        \fl(J_2(\mathbf A)) = J_2(\mathbf A + \mathbf{\delta A}),
        \quad \text{where} \quad \frac{|\delta \mathbf A_{ij}|}{|\mathbf A_{ij}|} \leq C \epsmach + \mathcal O (\epsmach^2),
        \label{eq:J2-bwd-error}
    \end{equation}
    for all $i,j$, with $C \approx 5$.
\end{proof}

\begin{theorem}{}
    \label{thm:J2-fwd-stability}
    Algorithm~\ref{alg:J2-stable} is forward stable in the sense that the absolute forward error satisfies
    \begin{equation}
        \abs{\fl(J_2) - J_2} \leq C \norm{\dev(\mathbf{A})}^2 \epsmach + \mathcal O (\epsmach^2).
    \end{equation}
\end{theorem}

\begin{proof}
    This is a consequence of the Jacobian and the backward stability result.
    Combining Lemmas~\ref{lem:J2-condition}
    and~\ref{lem:J2-backward}, we have
    \begin{equation}
        \begin{aligned}
            \abs{\fl(J_2) - J_2} & = \abs{J_2(\mathbf{A} + \mathbf{\delta A}) - J_2(\mathbf{A})}                                                & (\text{Lemma } \ref{lem:J2-backward})                   \\
                                 & = \abs{\inner{\dev(\mathbf{A})^\T}{\delta \mathbf{A}} + \mathcal O(\norm{\delta \mathbf{A}}^2)}              & (\text{Taylor expansion, Lemma \ref{lem:J2-condition}}) \\
                                 & = \abs{\inner{\dev(\mathbf{A})^\T}{\dev(\delta \mathbf{A})} + \mathcal O(\norm{\delta \mathbf{A}}^2)}                                                                  \\
                                 & \leq \norm{\dev(\mathbf{A})^\T}_2 \norm{\dev(\delta \mathbf{A})}_2 + \mathcal O(\norm{\delta \mathbf{A}}^2). & (\text{Cauchy--Schwarz}) \label{eq:dp-proof-last}
        \end{aligned}
    \end{equation}
    
    At this point, we need to show that the componentwise relative backward error bound from Eq.~\eqref{eq:J2-bwd-error} implies a normwise bound on the deviatoric part.
    This is not true in general, but we use the specific structure of the perturbation $\mathbf{\delta A}$ from
    Eq.~\eqref{eq:delta-A}. First, we notice that $\norm{\diag(\dev(\delta \mathbf{A}))} = \abs{\alpha}
        \norm{\diag(\dev(\mathbf{A}))}$, where $\diag(\cdot)$ denotes the diagonal part of a matrix. In addition,
    for any matrix,
    \begin{equation}
        \norm{\mathbf{B}}_F^2 = \norm{\diag(\mathbf{B})}_F^2 + \norm{\offdiag(\mathbf{B})}_F^2,
    \end{equation}
    since the diagonal and off-diagonal parts are orthogonal in the Frobenius inner product. We can write
    \begin{equation}
        \begin{aligned}
            \norm{\dev(\delta \mathbf{A})}_F^2 & = \norm{\diag(\dev(\delta \mathbf{A}))}_F^2 + \norm{\offdiag(\dev(\delta \mathbf{A}))}_F^2                                        \\
                                               & \leq \alpha^2 \norm{\diag(\dev(\mathbf{A}))}_F^2 + \max(\theta_5, \theta_5')^2 \norm{\offdiag(\dev(\mathbf{A}))}_F^2              \\
                                               & \leq \max(\alpha, \theta_5, \theta_5')^2 \left( \norm{\diag(\dev(\mathbf{A}))}_F^2 + \norm{\offdiag(\dev(\mathbf{A}))}_F^2\right) \\
                                               & = \max(\alpha, \theta_5, \theta_5')^2 \norm{\dev(\mathbf{A})}_F^2.
        \end{aligned}
    \end{equation}
    
    By norm equivalence in finite-dimensional spaces we obtain
    \begin{equation}
        \norm{\dev(\delta \mathbf{A})} \leq C \norm{\dev(\mathbf{A})} \, \epsmach + \mathcal O(\epsmach^2).
    \end{equation}
    Plugging this into Eq.~\eqref{eq:dp-proof-last} gives
    \begin{equation}
        \abs{\fl(J_2) - J_2} \leq C \norm{\dev(\mathbf{A})}^2 \epsmach + \mathcal O (\epsmach^2).
        \label{eq:J2-fwd-stability}
    \end{equation}
\end{proof}

\begin{lemma}{}
    \label{lem:dev-J2-estimate}
    Let $\mathbf{A} = \mathbf{U} \mathbf{D} \mathbf{U}^{-1}$ be a real, diagonalizable $3 \times 3$ matrix with real spectrum. Then
    \begin{equation}
        \frac{2}{9 \kappa_2^2}\, J_2 \;\leq\; \norm{\dev(\mathbf{A})}_F^2 \;\leq\; 18 \kappa_2^2\, J_2,
    \end{equation}
    where $\kappa_2 = \norm{\mathbf{U}}_2 \, \norm{\mathbf{U}^{-1}}_2$ is the spectral condition number of the matrix $\mathbf{U}$.
\end{lemma}

\begin{proof}
    Let $\mathbf{A} = \mathbf{U} \mathbf{D} \mathbf{U}^{-1}$ with $\mathbf
        D=\diag(\lambda_1,\lambda_2,\lambda_3)$ and $\lambda_i \in \mathbb R$. Denote the mean eigenvalue
    by $\bar\lambda = \frac{1}{3}\sum_{i=1}^3 \lambda_i = \frac{1}{3}\tr(\mathbf{A})$ and define the
    centered eigenvalues $\mu_i = \lambda_i - \bar\lambda$ (so $\sum_i \mu_i = 0$). The deviatoric
    part of $\mathbf{A}$ is
    \begin{equation}
        \begin{aligned}
            \mathbf{S} \coloneqq \dev(\mathbf{A})
            = \mathbf{A} - \bar\lambda \mathbf{I}
            = \mathbf{U}(\mathbf{D} - \bar\lambda \mathbf{I})\mathbf{U}^{-1}
            = \mathbf{U} \,\diag(\mu_1,\mu_2,\mu_3)\, \mathbf{U}^{-1}.
        \end{aligned}
    \end{equation}
    Using similarity invariance of the trace, we have
    \begin{equation}
        \begin{aligned}
            2 J_2
            = \tr(\mathbf{S}^2)
            = \tr\!\big(\mathbf{U} \diag(\mu)^2 \mathbf{U}^{-1}\big)
            = \tr\!\big(\diag(\mu)^2\big)
            = \sum_{i=1}^3 \mu_i^2
            = \norm{\diag(\mu)}_F^2
        \end{aligned}
    \end{equation}
    where $\diag(\mu) \coloneqq \diag(\mu_1,\mu_2,\mu_3)$. Hence
    \begin{equation}
        \norm{\diag(\mu)}_F = \sqrt{2 J_2}.
    \end{equation}
    For any matrices $\mathbf{A},\mathbf{X},\mathbf{B}$, the inequality
    \begin{equation}
        \norm{\mathbf{A} \mathbf{X} \mathbf{B}}_F \le \norm{\mathbf{A}}_F \, \norm{\mathbf{X}}_F\, \norm{\mathbf{B}}_F
    \end{equation}
    holds, since the Frobenius norm is submultiplicative. Applying this with $\mathbf{A}=\mathbf{U}$,
    $\mathbf{X}=\diag(\mu)$, and $\mathbf{B}=\mathbf{U}^{-1}$,
    \begin{equation}
        \begin{aligned}
            \norm{\mathbf{S}}_F
            = \norm{\mathbf{U} \diag(\mu) \mathbf{U}^{-1}}_F
            \le \norm{\mathbf{U}}_F \, \norm{\diag(\mu)}_F\, \norm{\mathbf{U}^{-1}}_F
            = 3 \kappa_2 \sqrt{2 J_2}.
        \end{aligned}
    \end{equation}
    We used the norm equivalence and the upper bound
    $\norm{\mathbf U}_F \leq \sqrt{3} \norm{\mathbf U}_2$ for any $3 \times 3$ matrix and the spectral norm $\norm{\cdot}_2$.
    Squaring gives the upper bound $\norm{\dev(\mathbf{A})}_F^2 \le 18 \kappa_2^2 J_2$.
    
    For the lower bound, rewrite $\diag(\mu) = \mathbf{U}^{-1} \mathbf{S} \mathbf{U}$ and apply the same inequality
    \begin{equation}
        \begin{aligned}
            \norm{\diag(\mu)}_F
            = \norm{\mathbf{U}^{-1} \mathbf{S} \mathbf{U}}_F
            \le \norm{\mathbf{U}^{-1}}_F \, \norm{\mathbf{S}}_F\, \norm{\mathbf{U}}_F
            = 3 \kappa_2 \norm{\mathbf{S}}_F.
        \end{aligned}
    \end{equation}
    Hence
    \begin{equation}
        \norm{\mathbf{S}}_F \ge \frac{\norm{\diag(\mu)}_F}{3 \kappa_2} = \frac{\sqrt{2 J_2}}{3 \kappa_2}
    \end{equation}
    and squaring yields
    \begin{equation}
        \norm{\dev(\mathbf{A})}_F^2 \ge \frac{2}{9 \kappa_2^2} J_2.
    \end{equation}
\end{proof}

\begin{corollary}{}
    \label{cor:J2-fwd-error-diagonalizable}
    For a real, diagonalizable $3 \times 3$ matrix $\mathbf{A} = \mathbf{U} \mathbf{D} \mathbf{U}^{-1}$ with real spectrum, Algorithm \ref{alg:J2-stable} satisfies
    \begin{equation}
        \abs{\fl(J_2) - J_2} \;\leq\; C\, \kappa_2^2\, J_2\, \epsmach \;+\; \mathcal O(\epsmach^2),
    \end{equation}
    where $\kappa_2 = \norm{\mathbf{U}}_2\, \norm{\mathbf{U}^{-1}}_2$ is the spectral condition number of $\mathbf{U}$.
    In particular, if $\mathbf{A}$ is symmetric (such that $\kappa_2 = 1$), the algorithm is accurate.
\end{corollary}

\begin{proof}
    This is a consequence of Lemma \ref{lem:dev-J2-estimate} and Theorem \ref{thm:J2-fwd-stability}.
\end{proof}

\emph{Remark (nonnormality and Henrici).}
For nonnormal matrices, Henrici's departure from normality considers the Schur form $\mathbf{A} =
    \mathbf{Q}(\mathbf{D} + \mathbf{N})\mathbf{Q}^{T}$ with $\mathbf{Q}$ orthogonal, $\mathbf{D}$
(block-)diagonal, and $\mathbf{N}$ strictly upper triangular, and defines the nonnormality measure as
$\nu(\mathbf{A}) \coloneqq \norm{\mathbf{N}}_F$, so that $\nu(\mathbf{A})=0$ iff $\mathbf{A}$ is normal. In
practice, large $\nu(\mathbf{A})$ is often accompanied by ill-conditioned eigenvectors (large
$\kappa=\norm{\mathbf{U}}\,\norm{\mathbf{U}^{-1}}$), which explains the $\kappa^2$ amplification
appearing in our bounds.
\footnote{\url{https://nhigham.com/2020/11/24/what-is-a-nonnormal-matrix/}}

\begin{figure}[htbp]
    \centering
    \begin{subfigure}[t]{0.49\linewidth}
        \centering
        \ErrorPlot{results/invariants-double_lim_J3J2-u1.dat}{J2}
        \caption{Forward error for the benchmark case in Fig.~\ref{fig:J2-J3-benchmark-d2} with transformation matrix $\mathbf U_1$ (well-conditioned, $\kappa_2 = 2$).}
        \label{fig:J2-error-d2-u1}
    \end{subfigure}
    \hfill
    \begin{subfigure}[t]{0.49\linewidth}
        \centering
        \ErrorPlot{results/invariants-single_lim_disc_t-u1.dat}{J2}
        \caption{Forward error for the benchmark case in Fig.~\ref{fig:J2-J3-benchmark-d1} with transformation matrix $\mathbf U_1$ (well-conditioned, $\kappa_2 = 2$).}
        \label{fig:J2-error-d1-u1}
    \end{subfigure}
    
    \begin{subfigure}[t]{0.49\linewidth}
        \ErrorPlot{results/invariants-double_lim_J3J2-symm.dat}{J2}
        \caption{Forward error for the benchmark case in Fig.~\ref{fig:J2-J3-benchmark-d2} with transformation matrix $\mathbf U_\text{symm}$ (orthogonal, $\kappa_2 = 1$).}
        \label{fig:J2-error-d2-symm}
    \end{subfigure}
    \hfill
    \begin{subfigure}[t]{0.49\linewidth}
        \ErrorPlot{results/invariants-single_lim_disc_t-symm.dat}{J2}
        \caption{Forward error for the benchmark case in Fig.~\ref{fig:J2-J3-benchmark-d1} with transformation matrix $\mathbf U_\text{symm}$ (orthogonal, $\kappa_2 = 1$).}
        \label{fig:J2-error-d1-symm}
    \end{subfigure}
    
    \begin{subfigure}[t]{0.49\linewidth}
        \ErrorPlot{results/invariants-double_lim_J3J2-u2.dat}{J2}
        \caption{Forward error for the benchmark case in Fig.~\ref{fig:J2-J3-benchmark-d2} with transformation matrix $\mathbf U_2(\gamma)$ (ill-conditioned eigenbasis).}
        \label{fig:J2-error-d2-u2}
    \end{subfigure}
    \hfill
    \begin{subfigure}[t]{0.49\linewidth}
        \ErrorPlot{results/invariants-single_lim_disc_t-u2.dat}{J2}
        \caption{Forward error for the benchmark case in Fig.~\ref{fig:J2-J3-benchmark-d1} with transformation matrix $\mathbf U_2(\gamma)$ (ill-conditioned eigenbasis).}
        \label{fig:J2-error-d1-u2}
    \end{subfigure}
    
    \caption{Numerical stability analysis for the invariant $J_2$.}
    \label{fig:J2-analysis}
\end{figure}

\emph{Remark (relative deviatoric conditioning).}
The relative condition number as defined in Eq.~\eqref{eq:def-rel-condition-number} is not informative for the
invariant $J_2$. For $\mathbf{A} = \diag(1,1,1+\delta)$, as $\delta \to 0$ we have
\begin{equation}
    \kappa_{J_2}(\mathbf{A})
    = \norm{\mathbf{J}_{J_2}} \frac{\norm{\mathbf{A}}}{\abs{J_2}}
    = \frac{\norm{\dev(\mathbf{A})}_F}{\frac{1}{2}\norm{\dev(\mathbf{A})}_F^2} \norm{\mathbf{A}}_F
    = \frac{2 \norm{\mathbf{A}}_F}{\norm{\dev(\mathbf{A})}_F}  \to \infty,
\end{equation}
where we used Lemma \ref{lem:J2-condition} and, for symmetric $\mathbf{A}$, $2J_2 = \norm{\dev(\mathbf{A})}_F^2$.
Nevertheless, Corollary \ref{cor:J2-fwd-error-diagonalizable} shows
that the algorithm is accurate for this symmetric case. Inspecting the proof of Theorem
\ref{thm:J2-fwd-stability} reveals that it is the deviatoric part of the perturbation that must be
controlled, rather than the full perturbation $\delta \mathbf{A}$. Motivated by this we define the
\emph{relative deviatoric condition number}
\begin{equation}
    \kappa_f^{\dev}(\mathbf{A}) \coloneqq \norm{\mathbf{J}_f} \frac{\norm{\dev(\mathbf{A})}}{\norm{f(\mathbf{A})}}.
\end{equation}
For $f=J_2$ and symmetric $\mathbf{A}$,
\begin{equation}
    \kappa_{J_2}^{\dev}(\mathbf{A})
    = \frac{\norm{\dev(\mathbf{A})}_F}{\frac{1}{2}\norm{\dev(\mathbf{A})}_F^2} \norm{\dev(\mathbf{A})}_F
    = 2, \quad (\text{for symmetric } \mathbf{A}),
\end{equation}
so $J_2$ is well-conditioned in the relative deviatoric sense for symmetric matrices.

The backward stability notion used to derive the forward error bound also needs refinement. What is
required in the proof of Theorem~\ref{thm:J2-fwd-stability} (in the first-order term) is
\begin{equation}
    \frac{\norm{\dev(\delta \mathbf{A})}}{\norm{\dev(\mathbf{A})}} \le C \epsmach,
\end{equation}
which we term \emph{relative deviatoric backward stability}. This notion is neither stronger nor
weaker than componentwise relative backward stability. We demonstrate this by two examples.

First, a perturbation that is componentwise relative stable but not relative deviatoric stable
\begin{equation}
    \delta \mathbf{A} = \begin{bmatrix}
        0 & 0 & 0         \\
        0 & 0 & 0         \\
        0 & 0 & \epsmach
    \end{bmatrix}
    \quad \text{for} \quad
    \mathbf{A} = \begin{bmatrix}
        1 & 0 & 0  \\
        0 & 1 & 0  \\
        0 & 0 & 1
    \end{bmatrix}
\end{equation}
has clearly each component small relative to $\mathbf{A}$, but $\norm{\dev(\delta \mathbf{A})}$
is of order $\epsmach$ while $\norm{\dev(\mathbf{A})}=0$.

Second, a perturbation that is relative deviatoric stable but not componentwise stable
\begin{equation}
    \delta \mathbf{A} = \mathbf{I} \quad \text{for} \quad \mathbf{A} = \mathbf{I}
\end{equation}
has $\norm{\dev(\delta \mathbf{A})} = 0$ so the relative deviatoric condition is satisfied,
but the componentwise relative error is of order 1.

\begin{corollary}{}
    \label{cor:J2-rel-bwd-stability}
    Algorithm \ref{alg:J2-stable} is backward stable in the relative deviatoric sense.
\end{corollary}
\begin{proof}
    This is a consequence of the proof of Lemma~\ref{lem:J2-backward} and the discussion in the proof
    of Theorem~\ref{thm:J2-fwd-stability}.
\end{proof}

\emph{Remark (intuition behind Algorithm \ref{alg:J2-stable}).} Algorithm \ref{alg:J2-stable} evaluates
$J_2$ by first forming the diagonal differences. This step is crucial for numerical stability
near $J_2=0$. In particular, it guarantees
\begin{equation}
    \fl(J_2(\alpha \mathbf{I})) = 0,
\end{equation}
so the algorithm is exact for the scaled identity matrix. In fact, for this to hold we need to show that the quadratic
term in Corollary \ref{cor:J2-fwd-error-diagonalizable} vanishes as well. This is a consequence of the second
directional derivative of $J_2$ (i.e., action of the Hessian) being
\begin{equation}
    \frac{\partial^2}{\partial \mathbf{A}^2} J_2 [\delta \mathbf{A}, \delta \mathbf{A}] =
    \inner{\dev(\delta \mathbf{A})^\T}{\dev(\delta \mathbf{A})},
\end{equation}
but the relative deviatoric backward stability then implies that
\begin{equation}
    \abs{\inner{\dev(\delta \mathbf{A})^\T}{\dev(\delta \mathbf{A})}} \leq C \norm{\dev(\mathbf{A})}^2 \epsmach^2
\end{equation}
so the quadratic term vanishes for $\mathbf{A} = \alpha \mathbf{I}$.

As seen in Theorem~\ref{thm:J2-fwd-stability}, this behavior is a necessary consequence of any algorithm that is
backward stable in the relative deviatoric sense. Moreover, in the vicinity of the scaled identity,
for example for some $\mathbf{A} = \alpha \mathbf{I} + \mathbf{E}$ where $\mathbf{E}$ is elementwise of
order $\epsmach$, the diagonal differences and the off-diagonal products are all of order
$\epsmach$. This prevents catastrophic cancellation and leads to the expression for $J_2$ that is of
order $\epsmach^2$.

\subsection{Discussion of the numerical benchmarks}

There are three different implementations of evaluation of $J_2$ benchmarked in Fig.~\ref{fig:J2-analysis}:
\textit{naive}, \textit{naive tensor}, and \textit{present}.

\emph{Naive} approach is based on Algorithm~\ref{alg:J2-naive}, which is an unrolled polynomial expression (monomial sum).
There is no structure-exploiting rearrangement of terms, so this algorithm is expected to be numerically unstable.
The second implementation is called \emph{naive tensor} and is based on the definition of $J_2$ as
$J_2 = \frac{1}{2} \tr(\dev(\mathbf{A})^2)$, where all operations are computed via a tensor
implementation in NumPy. The algorithm is listed in Algorithm~\ref{alg:J2-tensor}, where the trace is computed
using \texttt{numpy.trace} \citep{Numpy2025Trace} and matrix multiplication using
\texttt{numpy.matmul} \citep{Numpy2025Matmul}.

\begin{minipage}{1.0\linewidth}
    \begin{algorithm}[H]
        \begin{algorithmic}
            \Require $\mathbf{A} \in \mathbb{R}^{3 \times 3}$
            \State $J_2 = \frac{1}{3} (A_{00}^2 - A_{00} A_{11} - A_{00} A_{22} + 3 A_{01} A_{10} + 3 A_{02} A_{20} + A_{11}^2 - A_{11} A_{22} + 3 A_{12} A_{21} + A_{22}^2)$
            \State \Return $J_2$
        \end{algorithmic}
        \caption{\emph{Naive} evaluation of the invariant $J_2$}
        \label{alg:J2-naive}
    \end{algorithm}
\end{minipage}

\begin{minipage}{1.0\linewidth}
    \begin{algorithm}[H]
        \begin{algorithmic}
            \Require $\mathbf{A} \in \mathbb{R}^{3 \times 3}$
            \State $\mathbf S = \mathbf A - \frac{1}{3} \tr(\mathbf{A}) \mathbf I$ \Comment{Deviatoric part}
            \State $J_2 = \frac{1}{2} \tr(\mathbf S^2)$ \Comment{Matrix multiplication and trace}
            \State \Return $J_2$
        \end{algorithmic}
        \caption{\emph{Naive tensor} evaluation of the invariant $J_2$}
        \label{alg:J2-tensor}
    \end{algorithm}
\end{minipage}

The \textit{naive} implementation shows the largest forward errors in all benchmark cases, as seen in Fig.~\ref{fig:J2-analysis}.

The \textit{naive tensor} implementation is more accurate than the naive one, but still shows large forward
errors, especially in Figs.~\ref{fig:J2-error-d2-u1} and \ref{fig:J2-error-d2-symm}. The reason is that the
deviatoric part is computed based on the trace shift. Computation of the trace introduces rounding
errors, which then prevent the deviatoric part from being exactly zero even for the scaled identity
matrix.

Lastly, results for the \textit{present} algorithm (Algorithm~\ref{alg:J2-stable}) are included. This algorithm shows the
best accuracy in all benchmarks. In all cases the stability bound from Theorem
\ref{thm:J2-fwd-stability} is satisfied. This is true even for the most challenging case of the
transformation matrix being nonorthogonal and nearly singular, $\mathbf{U} = \mathbf{U}_2(\gamma)$. The
$\gamma$ parameter was chosen as $\gamma = 10^{-3}$, which leads to condition number
$\kappa_2(\mathbf{U}_2) \approx 9 \times 10^3$. The benchmark case of Fig.~\ref{fig:J2-J3-benchmark-d2}
has $J_2$ approaching zero, so in order to achieve the accuracy
promised by Corollary~\ref{cor:J2-fwd-error-diagonalizable}, the absolute forward error must decrease
proportionally. This is observed in Figs.~\ref{fig:J2-error-d2-u1}, \ref{fig:J2-error-d2-u2}
and \ref{fig:J2-error-d2-symm}. For the case of $\mathbf{D} = \diag(1, 1, 1 + \delta)$ and $\delta
    \approx 10^{-16}$ the exact value of $J_2 \approx \delta^{2} = 10^{-32}$. The accurate algorithm
must compute this value with relative error of order $\epsmach \approx 10^{-16}$, which means an
absolute error of order $10^{-48}$. This is indeed achieved in Fig.~\ref{fig:J2-error-d2-symm} and
for the well-conditioned case of Fig.~\ref{fig:J2-error-d2-u1}.

Note, that the included stability bound plots (solid lines) in Fig.~\ref{fig:J2-analysis} are based only on the
lowest order term from Eq.~\eqref{eq:J2-fwd-stability}, i.e., $\norm{\dev(\mathbf{A})}_F^2 \epsmach$. Following the
discussion in the remark about relative deviatoric conditioning, the higher order terms are proportional to
$\norm{\dev(\mathbf{A})}_F^2 \epsmach^2$ so they are negligible.

\section{Invariant $J_3$}

The invariant $J_3$ (see Eq.~\eqref{eq:principal-invariants}) is defined as the determinant of the deviatoric
part of a matrix. The algorithm presented in this section can be seen as a generalization of the algorithm in
\citet[\S 4.2.3 Eq. 16]{Harari2023Using} to nonsymmetric matrices.

\begin{lemma}
    \label{lem:J3-jacobian}
    The Jacobian of the $J_3$ invariant is given by
    \begin{equation}
        \mathbf J_{J_3} = \dev (\cof(\dev(\mathbf A))).
        \label{eq:J3-jacobian}
    \end{equation}
\end{lemma}

\begin{proof}
    Using
    \begin{equation}
        \frac{\partial}{\partial \mathbf A} \det(\mathbf A)[\delta \mathbf A] = \inner{\cof(\mathbf A)}{\delta \mathbf A}
    \end{equation}
    and the linearity of the deviatoric operator Eq.~\eqref{eq:dev-derivative}, we have
    \begin{equation}
        \begin{aligned}
            \frac{\partial}{\partial \mathbf A} J_3 [\delta \mathbf A] & = \frac{\partial}{\partial \mathbf A} \det(\dev(\mathbf A)) [\delta \mathbf A] \\
                                                                       & = \inner{\cof(\dev(\mathbf A))}{\dev(\delta \mathbf A)}                        \\
                                                                       & = \inner{\dev(\cof(\dev(\mathbf A)))}{\delta \mathbf A}.
        \end{aligned}
    \end{equation}
\end{proof}

Motivated by the observations in Section~\ref{sec:J2-stability}, we propose the following algorithm for evaluating $J_3$.

\begin{minipage}{1.0\linewidth}
    \begin{algorithm}[H]
        \begin{algorithmic}
            \Require $\mathbf A \in \mathbb{R}^{3 \times 3}$
            \State $d_{0} = A_{00} - A_{11}$, $d_{1} = A_{00} - A_{22}$, $d_{2} = A_{11} - A_{22}$ \Comment{Diagonal differences}
            \State $t_1 = d_1 + d_2$
            \State $t_2 = d_0 - d_2$
            \State $t_3 = -d_0 - d_1$
            \State $\text{offdiag} = A_{01} A_{12} A_{20} + A_{02} A_{10} A_{21}$ \Comment{Off-diagonal products}
            \State $\text{mixed} = \frac{1}{3} (A_{01} A_{10} t_1 + A_{02} A_{20} t_2 + A_{12} A_{21} t_3)$ \Comment{Mixed products}
            \State $\text{diag} = \frac{1}{27} t_1 t_2 t_3$ \Comment{Product of diagonal differences}
            \State $J_3 = \text{offdiag} + \text{mixed} - \text{diag}$
            \State \Return $J_3$
        \end{algorithmic}
        \caption{Evaluation of the $J_3$ invariant}
        \label{alg:J3-stable}
    \end{algorithm}
\end{minipage}

Algorithm \ref{alg:J3-stable} is an expansion of the determinant of the deviatoric part of
$\mathbf A$ expressed in terms of diagonal differences, off-diagonal products, and mixed products.
Similar to the $J_2$ invariant, $J_3$ approaches zero as the matrix
$\mathbf A$ approaches a scaled identity matrix. This is the reason for forming the diagonal
differences.

However, the $J_3$ invariant approaches zero in a more general case, when the deviatoric part of
$\mathbf A$ becomes singular. Consider an example matrix
\begin{equation}
    \mathbf A = \diag(1, 2, 3) = \begin{bmatrix}
        1 & 0 & 0  \\
        0 & 2 & 0  \\
        0 & 0 & 3
    \end{bmatrix}.
\end{equation}
This matrix is symmetric, but its deviatoric part is singular, i.e., $J_3 = \det(\dev(\mathbf A)) = 0$.
The diagonal differences in floating-point arithmetic are computed as
\begin{equation}
    \fl(d_0) = \fl(1 - 2) = -1 (1 + \delta_0), \quad \fl(d_1) = \fl(1 - 3) = -2 (1 + \delta_1), \quad \fl(d_2) = \fl(2 - 3) = -1 (1 + \delta_2),
\end{equation}
where $\abs{\delta_i} \leq \epsmach$. The diagonal combination $t_2$ is computed as
\begin{equation}
    \fl(t_2) = \fl(d_0 - d_2) = \fl(-1 (1 + \delta_0) + 1 (1 + \delta_2)) = (\delta_2 - \delta_0)(1 + \delta_3), \quad \abs{\delta_3} \leq \epsmach.
\end{equation}
This shows that the relative forward error $\abs{\fl(t_2) - t_2} / \abs{t_2}$ is unbounded. As a
consequence, Algorithm~\ref{alg:J3-stable} does not produce an exact zero for exactly singular
deviatoric matrices and cannot be considered accurate in those cases.

Assume we have an algorithm for $\fl(J_3)$ that is backward stable in the relative
deviatoric sense, i.e.,
\begin{equation}
    \fl(J_3(\mathbf A)) = J_3(\mathbf A + \delta \mathbf A), \quad \text{with} \quad \frac{\norm{\dev(\delta \mathbf A)}}{\norm{\dev(\mathbf A)}} \leq C \epsmach,
\end{equation}
for some constant $C$. Then we proceed similarly to the proof of Lemma~\ref{lem:J2-backward} and combine the backward error with the Jacobian
\begin{equation}
    \begin{aligned}
        \abs{\fl(J_3) - J_3} & = \abs{J_3(\mathbf A + \delta \mathbf A) - J_3(\mathbf A)}                                                      \\
                             & = \abs{\inner{\dev (\cof( \dev(\mathbf A)))}{\delta \mathbf A} + \mathcal O(\norm{\delta \mathbf A}^2)}         \\
                             & \leq \norm{\dev (\cof( \dev(\mathbf A)))} \norm{\dev(\delta \mathbf A)} + \mathcal O(\norm{\delta \mathbf A}^2) \\
                             & \leq C \norm{\dev (\cof( \dev(\mathbf A)))} \norm{\dev(\mathbf A)} \epsmach + \mathcal O(\epsmach^2).
        \label{eq:J3-fwd-error}
    \end{aligned}
\end{equation}

\begin{lemma}
    \label{lem:J3-fwd-stability}
    Any deviatoric backward stable algorithm for evaluating $J_3$ must be forward stable in the sense that the absolute forward error must satisfy
    \begin{equation}
        \abs{\fl(J_3) - J_3} \leq C \norm{\dev (\cof(\dev \mathbf A))} \norm{\dev(\mathbf A)} \epsmach + \mathcal O(\epsmach^2).
        \label{eq:J3-fwd-stability}
    \end{equation}
\end{lemma}

\begin{proof}
    Follows directly from Eq.~\eqref{eq:J3-fwd-error}.
\end{proof}

\begin{figure}[htbp]
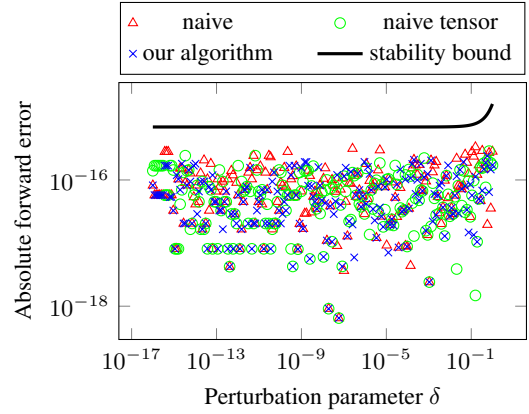
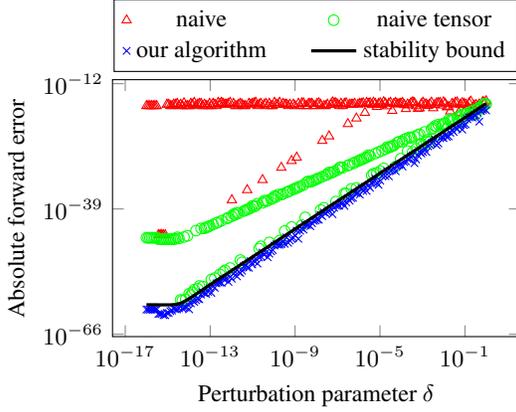
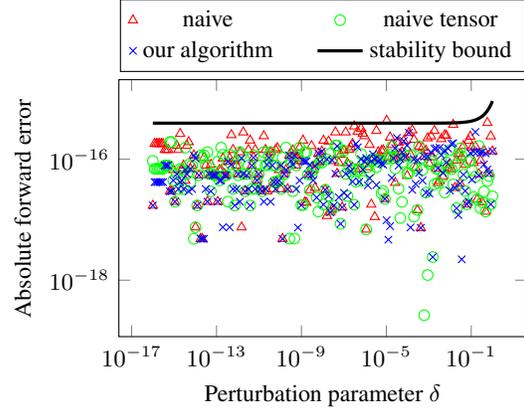
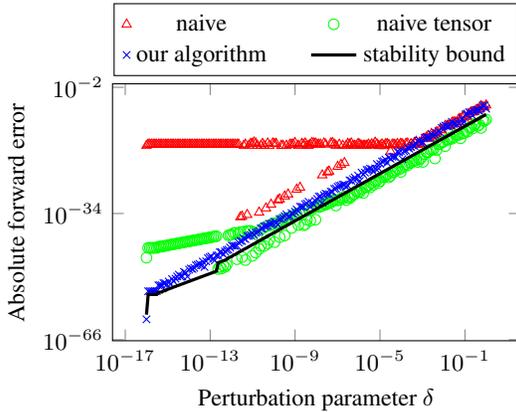
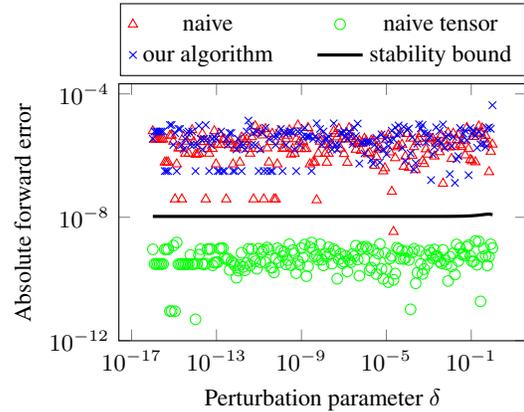

    \centering
    \begin{subfigure}[t]{0.49\linewidth}
        \centering
        \ErrorPlot{results/invariants-double_lim_J3J2-u1.dat}{J3}
        \caption{Forward error for the benchmark case in Fig. \ref{fig:J2-J3-benchmark-d2} with transformation matrix $\mathbf U_1$ (well-conditioned, $\kappa_2 = 2$).}
        \label{fig:J3-error-d2-u1}
    \end{subfigure}
    \hfill
    \begin{subfigure}[t]{0.49\linewidth}
        \centering
        \ErrorPlot{results/invariants-single_lim_disc_t-u1.dat}{J3}
        \caption{Forward error for the benchmark case in Fig. \ref{fig:J2-J3-benchmark-d1} with transformation matrix $\mathbf U_1$ (well-conditioned, $\kappa_2 = 2$).}
        \label{fig:J3-error-d1-u1}
    \end{subfigure}
    
    \begin{subfigure}[t]{0.49\linewidth}
        \ErrorPlot{results/invariants-double_lim_J3J2-symm.dat}{J3}
        \caption{Forward error for the benchmark case in Fig. \ref{fig:J2-J3-benchmark-d2} with transformation matrix $\mathbf U_\text{symm}$ (orthogonal, $\kappa_2 = 1$).}
        \label{fig:J3-error-d2-symm}
    \end{subfigure}
    \hfill
    \begin{subfigure}[t]{0.49\linewidth}
        \ErrorPlot{results/invariants-single_lim_disc_t-symm.dat}{J3}
        \caption{Forward error for the benchmark case in Fig. \ref{fig:J2-J3-benchmark-d1} with transformation matrix $\mathbf U_\text{symm}$ (orthogonal, $\kappa_2 = 1$).}
        \label{fig:J3-error-d1-symm}
    \end{subfigure}
    
    \begin{subfigure}[t]{0.49\linewidth}
        \ErrorPlot{results/invariants-double_lim_J3J2-u2.dat}{J3}
        \caption{Forward error for the benchmark case in Fig. \ref{fig:J2-J3-benchmark-d2} with transformation matrix $\mathbf U_2(\gamma)$ (ill-conditioned eigenbasis).}
        \label{fig:J3-error-d2-u2}
    \end{subfigure}
    \hfill
    \begin{subfigure}[t]{0.49\linewidth}
        \ErrorPlot{results/invariants-single_lim_disc_t-u2.dat}{J3}
        \caption{Forward error for the benchmark case in Fig. \ref{fig:J2-J3-benchmark-d1} with transformation matrix $\mathbf U_2(\gamma)$ (ill-conditioned eigenbasis).}
        \label{fig:J3-error-d1-u2}
    \end{subfigure}
    
    \caption{Numerical stability analysis for the invariant $J_3$.}
    \label{fig:J3-analysis}
\end{figure}

\subsection{Discussion of numerical benchmarks}

Three different implementations of $J_3$ are benchmarked in Fig.~\ref{fig:J3-analysis}: \textit{naive}, \textit{naive tensor}, and \textit{present}.

\emph{Naive} uses Algorithm~\ref{alg:J3-naive}, an unrolled polynomial expression (monomial sum) with no
structure-exploiting rearrangement of terms, so it is expected to be numerically unstable. The
second implementation, \emph{naive tensor}, is based on the definition $J_3 = \det(\dev(\mathbf A))$,
where all operations are computed via a tensor implementation in NumPy. The algorithm is listed in
Algorithm~\ref{alg:J3-tensor}, where the trace is computed using \texttt{numpy.trace}
\citep{Numpy2025Trace} and matrix multiplication using \texttt{numpy.matmul}
\citep{Numpy2025Matmul}.

\begin{minipage}{1.0\linewidth}
    \begin{algorithm}[H]
        \begin{algorithmic}
            \Require $\mathbf A \in \mathbb{R}^{3 \times 3}$
            \State $J_3 = \frac{1}{27} (2A_{00}^3 - 3A_{00}^2A_{11} - 3A_{00}^2A_{22} + 9A_{00}A_{01}A_{10}$
                \State $\qquad + 9A_{00}A_{02}A_{20} - 3A_{00}A_{11}^2 + 12A_{00}A_{11}A_{22} - 18A_{00}A_{12}A_{21}$
                \State $\qquad - 3A_{00}A_{22}^2 + 9A_{01}A_{10}A_{11} - 18A_{01}A_{10}A_{22} + 27A_{01}A_{12}A_{20}$
                \State $\qquad + 27A_{02}A_{10}A_{21} - 18A_{02}A_{11}A_{20} + 9A_{02}A_{20}A_{22} + 2A_{11}^3$
                \State $\qquad - 3A_{11}^2A_{22} + 9A_{11}A_{12}A_{21} - 3A_{11}A_{22}^2 + 9A_{12}A_{21}A_{22} + 2A_{22}^3)$
            \State \Return $J_3$
        \end{algorithmic}
        \caption{\emph{Naive} evaluation of the invariant $J_3$}
        \label{alg:J3-naive}
    \end{algorithm}
\end{minipage}

\begin{minipage}{1.0\linewidth}
    \begin{algorithm}[H]
        \begin{algorithmic}
            \Require $\mathbf A \in \mathbb{R}^{3 \times 3}$
            \State $\mathbf S = \mathbf A - \frac{1}{3} \tr(\mathbf A) \mathbf I$
            \State $J_3 = \det(\mathbf S)$
            \State \Return $J_3$
        \end{algorithmic}
        \caption{\emph{Naive tensor} evaluation of the invariant $J_3$}
        \label{alg:J3-tensor}
    \end{algorithm}
\end{minipage}

The \textit{naive} implementation shows the largest forward errors in all benchmark cases, as
seen in Fig.~\ref{fig:J3-analysis}. The error is so large that for $\delta \approx 10^{-16}$ and the
well-conditioned case in Fig.~\ref{fig:J3-error-d2-u1} the computed $J_3$ keeps an error of order
$10^{-16}$, which is 48 orders of magnitude larger than what a forward stable algorithm should produce.

The \textit{naive tensor} implementation is more accurate than the \textit{naive} one, but still shows large forward
errors, especially in Figs.~\ref{fig:J3-error-d2-u1} and \ref{fig:J3-error-d2-symm}, for the same reasons
as explained above for the $J_2$ invariant. It achieves the best accuracy for badly conditioned
cases in Figs.~\ref{fig:J3-error-d2-u2} and \ref{fig:J3-error-d1-u2}, probably due to the
implementation of the determinant in NumPy based on the numerically stable LU factorization \texttt{DGETRF}
from LAPACK \citep{Numpy2025Det,Anderson1999Lapack}.

For the \textit{present} algorithm (Algorithm~\ref{alg:J3-stable}), the well-conditioned cases with
$\mathbf U = \mathbf U_1$ (Figs.~\ref{fig:J3-error-d2-u1} and \ref{fig:J3-error-d1-u1}) and orthogonal cases with
$\mathbf U = \mathbf U_\text{symm}$ (Figs.~\ref{fig:J3-error-d2-symm} and \ref{fig:J3-error-d1-symm}) show that
the stability bound from Lem.~\ref{lem:J3-fwd-stability} is satisfied. This is not true for the most
challenging case of the transformation matrix being nonorthogonal and nearly singular,
$\mathbf U = \mathbf U_2(\gamma)$. The $\gamma$ parameter was chosen as $\gamma = 10^{-3}$, which leads to
condition number $\kappa_2(\mathbf U_2) \approx 9 \times 10^3$. In this case, the absolute forward
error exceeds the stability bound in Figs.~\ref{fig:J3-error-d2-u2} and \ref{fig:J3-error-d1-u2}.
This suggests that Algorithm~\ref{alg:J3-stable} is not backward stable in the relative deviatoric
sense for all matrices $\mathbf A$, but only conditionally stable, depending on the condition number
of the eigenbasis.

Note that the included stability bound plots (solid lines) in Fig. \ref{fig:J3-analysis} are based
only on the lowest order term from Eq.~\eqref{eq:J3-fwd-stability}, i.e.,
$\norm{\dev (\cof(\dev(\mathbf A)))}_F \norm{\dev(\mathbf A)}_F \epsmach$. It could be shown that the higher
order terms are negligible compared to the lowest order term in all benchmark cases.

\section{Discriminant $\Delta$}

\begin{lemma}
    \label{lem:disc-jacobian}
    The Jacobian of the discriminant $\Delta$ is given by
    \begin{equation}
        \mathbf J_{\Delta} = \dev ( 12 J_2^2 \mathbf A^\T - 54 J_3 \cof(\dev(\mathbf A))).
        \label{eq:disc-jacobian}
    \end{equation}
\end{lemma}

\begin{proof}
    Follows from the definition of the discriminant and the Jacobians of the invariants $J_2$, Eq.~\eqref{eq:J2-jacobian},
    and $J_3$, Eq.~\eqref{eq:J3-jacobian}. The deviatoric operator is then moved outside by linearity.
\end{proof}

Motivated by the observations in Section~\ref{sec:J2-stability} and our previous work, we propose
the following algorithm for evaluating $\Delta$. We briefly recall the main ideas from \citet{Habera2021Symbolic}.
We rely on an expression for the discriminant $\Delta$ as the determinant of a matrix
$\mathbf B \in \mathbb R^{3 \times 3}$ whose entries are invariants of powers of the matrix
$\mathbf A$, see \citet[Lemma 1.]{Parlett2002Discriminant}. Specifically, we have
\begin{equation}
    \Delta = \det(\mathbf B), \quad \text{where} \quad B_{ij} = \inner{\mathbf A^{i-1}}{(\mathbf A^{j-1})^\T}
    \label{eq:disc-determinant}
\end{equation}
for $i, j = 1, 2, 3$. In addition, the matrix $\mathbf B = \mathbf X \mathbf Y$ exhibits a factorization into two matrices
$\mathbf X \in \mathbb R^{3 \times 9}$ and $\mathbf Y \in \mathbb R^{9 \times 3}$, which are built
from column- and row-stacked powers of $\mathbf A$, respectively. Using the Cauchy--Binet formula,
the determinant $\det(\mathbf X \mathbf Y)$ can be expressed as a sum of 14 condensed terms,
which we term the \emph{sum-of-products} formula,
\begin{equation}
    \Delta = \sum_{i=1}^{14} w_i u_i v_i,
    \label{eq:sop-disc}
\end{equation}
where $\mathbf u = (u_1, \dots, u_{14})$ and $\mathbf v = (v_1, \dots, v_{14})$ are auxiliary
vectors built from products of entries of $\mathbf A$ and its transpose, respectively, and
$\mathbf w = (w_1, \dots, w_{14})$ is a vector of integer weights. The important property of the
sum-of-products formula is that both factors $\mathbf u$ and $\mathbf v$ approach zero as the matrix $\mathbf A$
approaches a matrix with multiple eigenvalues. This is related to the Cayley--Hamilton theorem and the
fact that the columns of $\mathbf X$ and rows of $\mathbf Y$ become linearly dependent in such
situations, since $\mathbf A$ satisfies its own minimal polynomial.

\begin{minipage}{1.0\linewidth}
    \begin{algorithm}[H]
        \begin{algorithmic}
            \Require $\mathbf A \in \mathbb{R}^{3 \times 3}$
            \State $d_{0} = A_{00} - A_{11}$, $d_{1} = A_{00} - A_{22}$, $d_{2} = A_{11} - A_{22}$ \Comment{Diagonal differences}
            \State $\mathbf u = \Call{dx}{\mathbf A, d_0, d_1, d_2}$  \Comment{Auxiliary vector}
            \State $\mathbf v = \Call{dx}{\mathbf A^\T, d_0, d_1, d_2}$  \Comment{Auxiliary vector}
            \State $\mathbf w = (9, 6, 6, 6, 8, 8, 8, 2, 2, 2, 2, 2, 2, 1)$ \Comment{Weights}
            \State $\Delta = \sum_{i=1}^{14} w_i\, u_i\, v_i$ \Comment{Sum of products}
            \State \Return $\Delta$
            \Function{dx}{$\mathbf M, d_0, d_1, d_2$}
            \State $r_1 = M_{01} M_{12} M_{20} - M_{02} M_{10} M_{21}$
            \State $r_2 = -M_{01} M_{02} d_{2} + M_{01}^2 M_{12} - M_{02}^2 M_{21}$
            \State $r_3 = M_{01} M_{21} d_{1} - M_{01}^2 M_{20} + M_{02} M_{21}^2$
            \State $r_4 = M_{02} M_{12} d_{0} + M_{01} M_{12}^2 - M_{02}^2 M_{10}$
            \State $r_5 = M_{01} M_{12} d_{1} - M_{01} M_{02} M_{10} + M_{02} M_{12} M_{21}$
            \State $r_6 = M_{02} M_{21} d_{0} - M_{01} M_{02} M_{20} + M_{01} M_{12} M_{21}$
            \State $r_7 = -M_{02} M_{10} d_{2} + M_{01} M_{10} M_{12} - M_{02} M_{12} M_{20}$
            \State $r_8 = M_{12} d_{0} d_{1} - M_{02} M_{10} d_{1} + M_{01} M_{10} M_{12} - M_{12}^2 M_{21}$
            \State $r_9 = M_{12} d_{0} d_{1} - M_{02} M_{10} d_{0} + M_{02} M_{12} M_{20} - M_{12}^2 M_{21}$
            \State $r_{10} = M_{01} d_{1} d_{2} - M_{02} M_{21} d_{2} + M_{01} M_{02} M_{20} - M_{01}^2 M_{10}$
            \State $r_{11} = M_{01} d_{1} d_{2} + M_{02} M_{21} d_{1} + M_{01} M_{12} M_{21} - M_{01}^2 M_{10}$
            \State $r_{12} = -M_{02} d_{0} d_{2} + M_{01} M_{12} d_{0} + M_{02} M_{12} M_{21} - M_{02}^2 M_{20}$
            \State $r_{13} = M_{02} d_{0} d_{2} + M_{01} M_{12} d_{2} - M_{01} M_{02} M_{10} + M_{02}^2 M_{20}$
            \State $r_{14} = d_{0} d_{1} d_{2} - M_{01} M_{10} d_{0} + M_{02} M_{20} d_{1} - M_{12} M_{21} d_{2}$
            \State $\mathbf r = (r_1,\dots,r_{14})$
            \State \Return $\mathbf r$
            \EndFunction
        \end{algorithmic}
        \caption{Evaluation of the discriminant invariant $\Delta$}
        \label{alg:disc-stable}
    \end{algorithm}
\end{minipage}

Algorithm \ref{alg:disc-stable} implements the sum-of-products formula \eqref{eq:sop-disc} for
evaluating $\Delta$. In addition to the Cauchy--Binet factorization, it incorporates the computation
of diagonal differences to improve numerical stability near $J_2 = 0$, similar to the stable
algorithms for $J_2$ and $J_3$. Note that the individual factors $r_i$ in the auxiliary vectors
$\mathbf u$ and $\mathbf v$ contain the diagonal elements of the matrix $\mathbf A$ only in the
form of their differences.

\begin{lemma}
    \label{lem:disc-fwd-stability}
    Any deviatoric backward stable algorithm for evaluating $\Delta$ must be forward stable in the
    sense that the absolute forward error must satisfy
    \begin{equation}
        \abs{\fl(\Delta) - \Delta} \leq C \norm{\dev ( 12 J_2^2 \mathbf A^\T - 54 J_3 \cof(\dev(\mathbf A)))} \norm{\dev(\mathbf A)} \epsmach + \mathcal O(\epsmach^2).
    \end{equation}
\end{lemma}

\begin{proof}
    This follows directly from the backward error analysis and Eq.~\eqref{eq:disc-jacobian}.
\end{proof}
\begin{figure}[htbp]
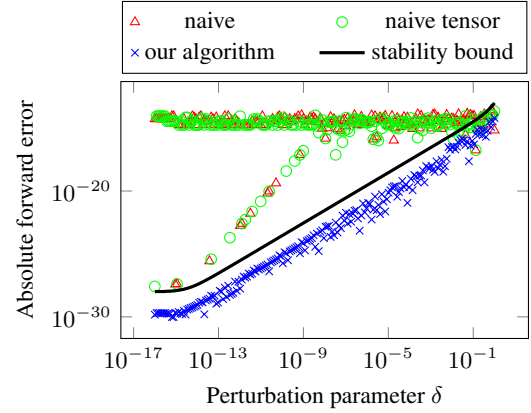
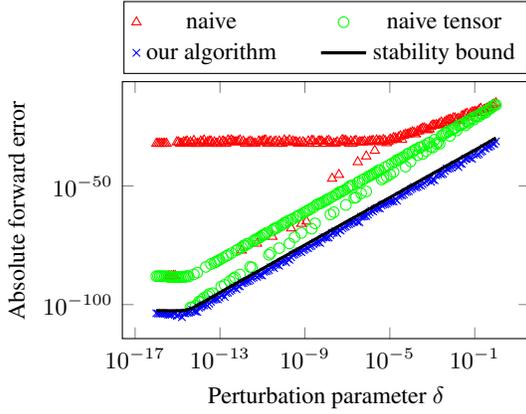
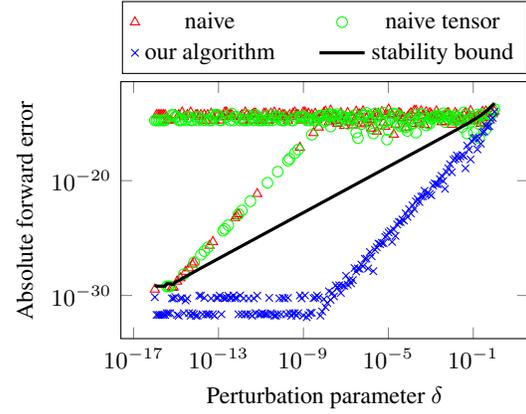
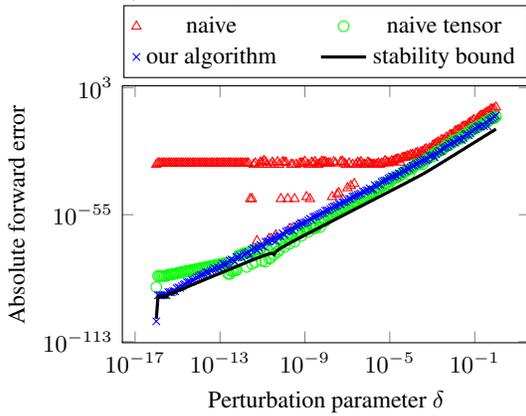
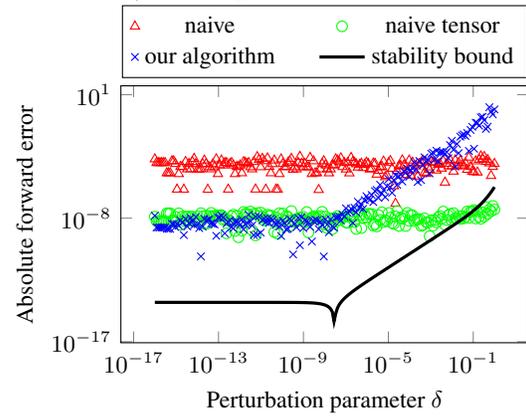

    \centering
    \begin{subfigure}[t]{0.49\linewidth}
        \centering
        \ErrorPlot{results/invariants-double_lim_J3J2-u1.dat}{disc}
        \caption{Forward error for the benchmark case in Fig.~\ref{fig:J2-J3-benchmark-d2} with transformation matrix
            $\mathbf U_1$ (well-conditioned, $\kappa_2 = 2$).}
        \label{fig:disc-error-d2-u1}
    \end{subfigure}
    \hfill
    \begin{subfigure}[t]{0.49\linewidth}
        \centering
        \ErrorPlot{results/invariants-single_lim_disc_t-u1.dat}{disc}
        \caption{Forward error for the benchmark case in Fig.~\ref{fig:J2-J3-benchmark-d1} with
            transformation matrix $\mathbf U_1$ (well-conditioned, $\kappa_2 = 2$).}
        \label{fig:disc-error-d1-u1}
    \end{subfigure}
    
    \begin{subfigure}[t]{0.49\linewidth}
        \ErrorPlot{results/invariants-double_lim_J3J2-symm.dat}{disc}
        \caption{Forward error for the benchmark case in Fig.~\ref{fig:J2-J3-benchmark-d2} with
            transformation matrix $\mathbf U_\text{symm}$ (orthogonal, $\kappa_2 = 1$).}
        \label{fig:disc-error-d2-symm}
    \end{subfigure}
    \hfill
    \begin{subfigure}[t]{0.49\linewidth}
        \ErrorPlot{results/invariants-single_lim_disc_t-symm.dat}{disc}
        \caption{Forward error for the benchmark case in Fig.~\ref{fig:J2-J3-benchmark-d1} with
            transformation matrix $\mathbf U_\text{symm}$ (orthogonal, $\kappa_2 = 1$).}
        \label{fig:disc-error-d1-symm}
    \end{subfigure}
    
    \begin{subfigure}[t]{0.49\linewidth}
        \ErrorPlot{results/invariants-double_lim_J3J2-u2.dat}{disc}
        \caption{Forward error for the benchmark case in Fig.~\ref{fig:J2-J3-benchmark-d2} with transformation matrix
            $\mathbf U_2(\gamma)$ (ill-conditioned eigenbasis).}
        \label{fig:disc-error-d2-u2}
    \end{subfigure}
    \hfill
    \begin{subfigure}[t]{0.49\linewidth}
        \ErrorPlot{results/invariants-single_lim_disc_t-u2.dat}{disc}
        \caption{Forward error for the benchmark case in Fig.~\ref{fig:J2-J3-benchmark-d1} with
            transformation matrix $\mathbf U_2(\gamma)$ (ill-conditioned eigenbasis).}
        \label{fig:disc-error-d1-u2}
    \end{subfigure}
    
    \caption{Numerical stability analysis for the discriminant $\Delta$.}
    \label{fig:disc-analysis}
\end{figure}

\subsection{Discussion of numerical benchmarks}

Numerical benchmarks evaluating the absolute forward error of Algorithm \ref{alg:disc-stable} are
shown in Fig.~\ref{fig:disc-analysis}.

Similar to the $J_2$ and $J_3$ invariants, we benchmark three implementations for evaluating $\Delta$
in Fig.~\ref{fig:disc-analysis}. \textit{Naive} and \textit{naive tensor} implementations are based on
the direct evaluation of the formula $\Delta = 4 J_2^3 - 27 J_3^2$, where $J_2$ and $J_3$ are
computed using the \textit{naive} and \textit{naive tensor} algorithms, respectively.

The \textit{naive} implementation is clearly unstable in all benchmark cases, with the forward error exceeding
the stability bound by several orders of magnitude. The \textit{naive tensor} implementation is more stable,
but still fails to achieve the stability bound.

The proposed algorithm (Algorithm~\ref{alg:disc-stable}) is stable for benchmarks with a well-conditioned
eigenbasis (Figs.~\ref{fig:disc-error-d2-u1} and \ref{fig:disc-error-d1-u1}) and an orthogonal
eigenbasis (Figs.~\ref{fig:disc-error-d2-symm} and \ref{fig:disc-error-d1-symm}). In these cases,
the forward error is close to the stability bound. However, the forward error increases for the
benchmark with an ill-conditioned eigenbasis (Figs.~\ref{fig:disc-error-d2-u2} and
\ref{fig:disc-error-d1-u2}), exceeding the stability bound, and even those produced by the \textit{naive}
and \textit{naive tensor} algorithms.

Note that the included stability bound plots (solid lines) in Fig. \ref{fig:disc-analysis} are based
only on the lowest order term from Lemma~\ref{lem:disc-fwd-stability}, i.e.,
$\norm{\dev(12 J_2^2 \mathbf A^\T - 54 J_3 \cof(\dev(\mathbf A)))}_F \norm{\dev(\mathbf A)}_F \epsmach$. It could be shown that the higher
order terms are negligible compared to the lowest order term in all benchmark cases.

\section{Eigenvalues}
\label{sec:eigvals-stability}

For eigenvalues, there exists a classical perturbation result called the Bauer--Fike theorem \citep{Bauer1960Norms},
which provides an absolute bound on the perturbation of eigenvalues of diagonalizable matrices.

\begin{theorem}[Bauer--Fike, 1960]
    \label{thm:bauer-fike}
    Let $\mathbf A \in \mathbb C^{n \times n}$ be diagonalizable, i.e., $\mathbf A = \mathbf U \mathbf D \mathbf U^{-1}$,
    where $\mathbf D$ is diagonal and $\mathbf U$ is invertible. Let $\lambda$ be an eigenvalue of $\mathbf A$.
    Then there exists an eigenvalue $\tilde{\lambda}$ of $\mathbf A + \delta \mathbf A$ such that
    \begin{equation}
        |\tilde{\lambda} - \lambda| \leq \kappa_p(\mathbf U) \norm{\delta \mathbf A}_p,
    \end{equation}
    where $\kappa_p(\mathbf U) = \norm{\mathbf U}_p \norm{\mathbf U^{-1}}_p$ is the $p$-norm condition number of the eigenbasis.
\end{theorem}

\begin{proof}
    See, e.g., \citet[Thm. 7.2.2]{Golub2013Matrix}.
\end{proof}

We use the Bauer--Fike theorem to derive a stability bound for eigenvalue computation, summarized in
the following lemma.

\begin{lemma}
    \label{lem:eigvals-fwd-stability}
    Any backward stable algorithm for evaluating the eigenvalues of a real diagonalizable matrix $\mathbf
        A \in \mathbb R^{3 \times 3}$ with real spectrum must be forward stable in the sense that the
    absolute forward error must satisfy
    \begin{equation}
        \abs{\fl(\lambda) - \lambda} \leq C \kappa_2(\mathbf U) \norm{\mathbf A} \epsmach + \mathcal O(\epsmach^2),
    \end{equation}
    where $\lambda$ is an eigenvalue of $\mathbf A$, $\mathbf U$ is the eigenbasis of $\mathbf A$,
    and $C$ is a moderate constant.
\end{lemma}

\begin{proof}
    This is a consequence of the Bauer--Fike theorem, Theorem~\ref{thm:bauer-fike}, and the definition of backward stability,
    similar to Theorem \ref{thm:J2-fwd-stability} and Lemmas \ref{lem:J3-fwd-stability} and \ref{lem:disc-fwd-stability}.
    
    Backward stability implies that there exists a perturbation $\delta \mathbf A$ such that
    \begin{equation}
        \fl(\lambda(\mathbf A)) = \lambda(\mathbf A + \delta \mathbf A), \quad \text{with} \quad \frac{\norm{\delta \mathbf A}}{\norm{\mathbf A}} \leq C \epsmach.
    \end{equation}
    Using the Bauer--Fike theorem with, for example, $p = 2$, we have
    \begin{equation}
        \begin{aligned}
            \abs{\fl(\lambda) - \lambda} & = \abs{\lambda(\mathbf A + \delta \mathbf A) - \lambda(\mathbf A)}                        \\
                                         & \leq \kappa_2(\mathbf U) \norm{\delta \mathbf A}_2 + \mathcal{O}(\norm{\delta \mathbf A}) \\
                                         & \leq C \kappa_2(\mathbf U) \norm{\mathbf A} \epsmach + \mathcal O(\epsmach^2),
        \end{aligned}
    \end{equation}
    where norm equivalence justifies the final transition to an arbitrary matrix norm $\norm{\cdot}$.
\end{proof}

Having developed stable algorithms for the invariants $I_1, J_2, J_3,$ and $\Delta$, we can now propose an algorithm
for the eigenvalues based on the trigonometric formula Eq.~\eqref{eq:eigvals-trig} and $\arctan$ expression for the
triple-angle $\varphi$, Eq.~\eqref{eq:triple-angle-arctan}, summarized in Algorithm~\ref{alg:eigvals-stable}.

\begin{minipage}{1.0\linewidth}
    \begin{algorithm}[H]
        \begin{algorithmic}
            \Require $I_1, J_2, J_3$ and $\Delta$
            \State $t = \sqrt{27 \Delta} / (27 J_3)$ \Comment{Triple-angle argument}
            \State $\varphi = \arctan(t)$ \Comment{Triple-angle}
            \State $\lambda_k = \frac{1}{3} \left( I_1 + 2 \sqrt{3 J_2} \cos((\varphi + 2 \pi k) / 3) \right)$ for $k \in \{1, 2, 3\}$
            \State \Return $\lambda_1, \lambda_2, \lambda_3$
        \end{algorithmic}
        \caption{Evaluation of the eigenvalues}
        \label{alg:eigvals-stable}
    \end{algorithm}
\end{minipage}

\begin{figure}[htbp]
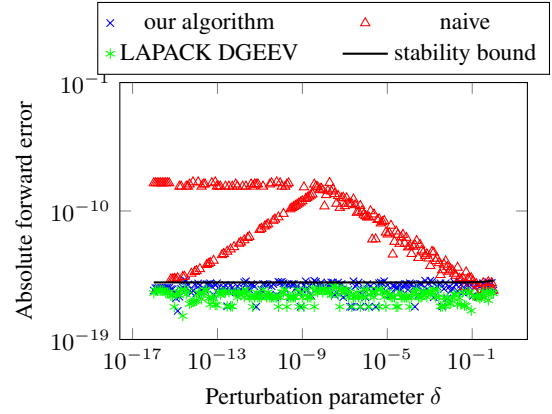
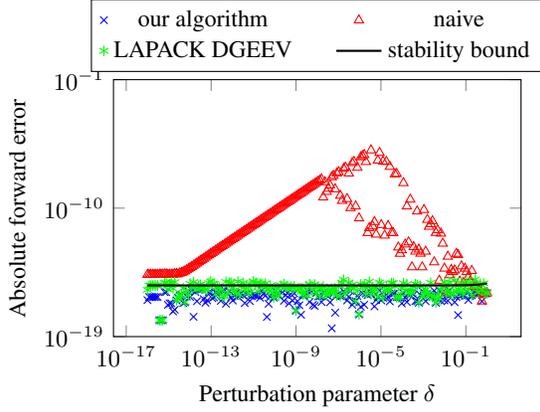
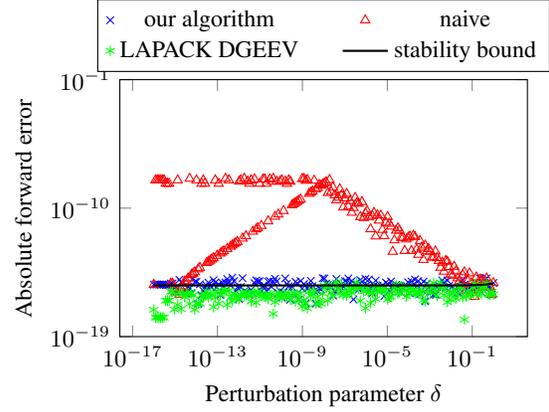
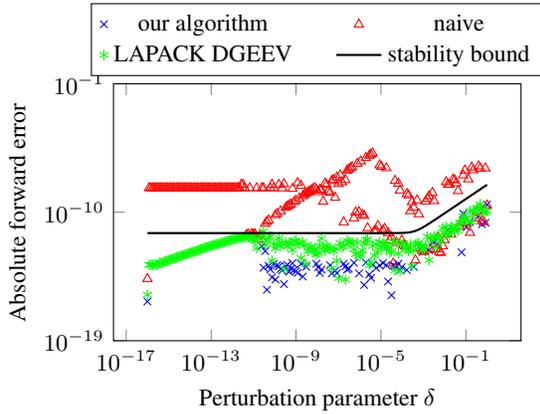
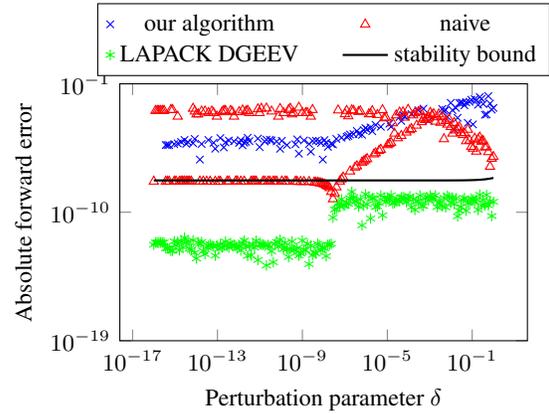

    \centering
    \begin{subfigure}[t]{0.49\linewidth}
        \centering
        \ErrorPlotEigvals{results/invariants-double_lim_J3J2-u1.dat}{eig2}{1e-19}{1e-1}
        \caption{Forward error for the benchmark case in Fig.~\ref{fig:J2-J3-benchmark-d2} with transformation matrix $\mathbf U_1$ (well-conditioned, $\kappa_2 = 2$).}
        \label{fig:eigvals-error-d2-u1}
    \end{subfigure}
    \hfill
    \begin{subfigure}[t]{0.49\linewidth}
        \centering
        \ErrorPlotEigvals{results/invariants-single_lim_disc_t-u1.dat}{eig2}{1e-19}{1e-1}
        \caption{Forward error for the benchmark case in Fig.~\ref{fig:J2-J3-benchmark-d1} with transformation matrix $\mathbf U_1$ (well-conditioned, $\kappa_2 = 2$).}
        \label{fig:eigvals-error-d1-u1}
    \end{subfigure}
    
    \begin{subfigure}[t]{0.49\linewidth}
        \ErrorPlotEigvals{results/invariants-double_lim_J3J2-symm.dat}{eig2}{1e-19}{1e-1}
        \caption{Forward error for the benchmark case in Fig.~\ref{fig:J2-J3-benchmark-d2} with transformation matrix $\mathbf U_\text{symm}$ (orthogonal, $\kappa_2 = 1$).}
        \label{fig:eigvals-error-d2-symm}
    \end{subfigure}
    \hfill
    \begin{subfigure}[t]{0.49\linewidth}
        \ErrorPlotEigvals{results/invariants-single_lim_disc_t-symm.dat}{eig2}{1e-19}{1e-1}
        \caption{Forward error for the benchmark case in Fig.~\ref{fig:J2-J3-benchmark-d1} with transformation matrix $\mathbf U_\text{symm}$ (orthogonal, $\kappa_2 = 1$).}
        \label{fig:eigvals-error-d1-symm}
    \end{subfigure}
    
    \begin{subfigure}[t]{0.49\linewidth}
        \ErrorPlotEigvals{results/invariants-double_lim_J3J2-u2.dat}{eig2}{1e-19}{1e-1}
        \caption{Forward error for the benchmark case in Fig.~\ref{fig:J2-J3-benchmark-d2} with transformation matrix $\mathbf U_2(\gamma)$ (ill-conditioned eigenbasis).}
        \label{fig:eigvals-error-d2-u2}
    \end{subfigure}
    \hfill
    \begin{subfigure}[t]{0.49\linewidth}
        \ErrorPlotEigvals{results/invariants-single_lim_disc_t-u2.dat}{eig2}{1e-19}{1e-1}
        \caption{Forward error for the benchmark case in Fig.~\ref{fig:J2-J3-benchmark-d1} with transformation matrix $\mathbf U_2(\gamma)$ (ill-conditioned eigenbasis).}
        \label{fig:eigvals-error-d1-u2}
    \end{subfigure}
    
    \caption{Numerical stability analysis for eigenvalue computation.}
    \label{fig:eigvals-analysis}
\end{figure}

\subsection{Discussion of numerical benchmarks}

Three different implementations of eigenvalue evaluation are benchmarked in Fig.~\ref{fig:eigvals-analysis}:
\textit{naive}, \textit{LAPACK DGEEV}, and \textit{present}.

\emph{Naive} approach is based on Algorithm~\ref{alg:eigvals-stable} with invariants $J_2, J_3$ and $\Delta$
computed with the naive algorithms, i.e., Algorithm~\ref{alg:J2-naive} and Algorithm~\ref{alg:J3-naive}.

\emph{LAPACK DGEEV} is based on the LAPACK library routine \texttt{DGEEV}, which computes all
eigenvalues and, optionally, the left and/or right eigenvectors of a real nonsymmetric matrix
\citep{Anderson1999Lapack}. We use the NumPy wrapper \texttt{numpy.linalg.eigvals}, see \citet{Numpy2025Eigvals},
for this routine.

The \textit{naive} approach is numerically unstable and produces forward errors as large as
$\epsmach^{1/3} \approx 10^{-5}$ for the benchmark case of a nearly triple eigenvalue in Figs.~
\ref{fig:eigvals-error-d2-u1}, \ref{fig:eigvals-error-d2-symm}, and \ref{fig:eigvals-error-d2-u2}.
For the benchmark case of a nearly double eigenvalue in Figs.~\ref{fig:eigvals-error-d1-u1},
\ref{fig:eigvals-error-d1-symm}, and \ref{fig:eigvals-error-d1-u2}, the forward error is as large
as $\epsmach^{1/2} \approx 10^{-8}$.

For well-conditioned cases with $\mathbf U = \mathbf U_1$ (Figs.~\ref{fig:eigvals-error-d2-u1} and
\ref{fig:eigvals-error-d1-u1}) and orthogonal cases with $\mathbf U = \mathbf U_\text{symm}$ (Figs.~
\ref{fig:eigvals-error-d2-symm} and \ref{fig:eigvals-error-d1-symm}), \textit{present} algorithm satisfies the
stability bound from Lemma~\ref{lem:eigvals-fwd-stability}.

\textit{LAPACK DGEEV} produces forward errors that are close to the stability bound in all benchmark cases,
even in the most challenging case of the transformation matrix being nonorthogonal and nearly
singular, $\mathbf U = \mathbf U_2(\gamma)$ (Figs.~\ref{fig:eigvals-error-d2-u2} and
\ref{fig:eigvals-error-d1-u2}). The $\gamma$ parameter was chosen as $\gamma = 10^{-3}$, which
leads to condition number $\kappa_2(\mathbf U_2) \approx 9 \times 10^3$.

\section{Performance benchmarks}

In this section, we present performance benchmarks of the proposed eigenvalue algorithm in comparison
with the numerical library LAPACK \citep{Anderson1999Lapack}. The benchmarks were executed on a MacBook Pro (2024)
with Apple M4 (ARM) CPU (10-core CPU, 120 GB/s memory bandwidth).

The benchmarks were written in C11 with LAPACK routine DGEEV called via the LAPACKE C interface version 3.12.1.
The OpenBLAS library version 0.3.29 was linked for BLAS and LAPACK functionality. This setup was run inside
a Docker container based on the official Python 3.14 Docker image \texttt{python:3.14-trixie}. The container
is pre-installed with GCC version 14.2.0. The code was compiled with optimization flags \texttt{-O3 -march=native}.

The benchmark consists of evaluating eigenvalues of an example real, diagonalizable $3 \times 3$
matrix. The matrix was generated as in Eq.~\eqref{eq:test-matrix} with transformation matrix
$\mathbf U = \mathbf U_1$ (well-conditioned case, $\kappa_2(\mathbf U_1) = 2$) and eigenvalues along
the benchmark path in Fig.~\ref{fig:J2-J3-benchmark-d2}, i.e., $\mathbf D = \diag(-1, 1, 1 + 10^{-14})$.
This is a challenging case of nearly a double eigenvalue, but both $J_2$ and $J_3$ remain finite
and away from zero. The test matrix reads explicitly as
\begin{align}
    \mathbf A = \mathbf U_1 \mathbf D \mathbf U_1^{-1} = \begin{bmatrix}
                                                             \fl(0)  & \fl(5 \cdot 10^{-15})     & \fl(1 + 5 \cdot 10^{-15}) \\
                                                             \fl(-1) & \fl(1 + 5 \cdot 10^{-15}) & \fl(1 + 5 \cdot 10^{-15}) \\
                                                             \fl(1)  & \fl(5 \cdot 10^{-15})     & \fl(5 \cdot 10^{-15})
                                                         \end{bmatrix}.
\end{align}

We performed a total of $10^6$ evaluations for the above matrix and measured
the total execution time for both the proposed algorithm and LAPACK DGEEV. The results are presented in
Table~\ref{tab:eigvals-benchmark}. The proposed algorithm is approximately ten times faster than LAPACK DGEEV
for this benchmark, while both methods returned eigenvalues with the expected absolute forward error
on the order of machine precision, i.e., approximately $10^{-16}$.

\begin{table}[htbp]
    \centering
    \caption{Performance comparison of eigenvalue computation for the test matrix $\mathbf A$ over $10^6$ evaluations.}
    \label{tab:eigvals-benchmark}
    \begin{tabular}{lcc}
        \hline
        Method            & Average time per evaluation $\pm$ std. [$\SI{}{\nano\second}$] & Fastest time per evaluation [$\SI{}{\nano\second}$] \\
        \hline
        Present algorithm & $38.2 \pm 1.2$                                                 & $35.02$                                             \\
        LAPACK DGEEV      & $396.4 \pm 4.3$                                                & $381.18$                                            \\
    \end{tabular}
\end{table}

For convenience, we also provide wrappers for Python using CFFI \citep{Rigo2025CFFI}.
Our implementation is available as part of the open-source library \texttt{eig3x3}, see \citet{Habera2025Eig3x3}.
The library \texttt{eig3x3} contains implementations of all algorithms presented in this paper, including naive
and stable algorithms for evaluating the invariants $J_2$, $J_3$, and $\Delta$, as well as eigenvalue computation
and the benchmarking code used to generate the results in this paper.
The eigenvalue algorithms are available as
\begin{itemize}
    \item \texttt{eig3x3.eigvals} for computing the eigenvalues of a real, diagonalizable
          $3 \times 3$ matrix,
    \item \texttt{eig3x3.eigvalss} for computing the eigenvalues of a real, symmetric
          $3 \times 3$ matrix.
\end{itemize}

Note that the C implementation of the proposed algorithm is not optimized for the specific
CPU architecture. Further optimizations, such as SIMD vectorization, could lead to even better
performance. On the other hand, LAPACK is a highly optimized library that benefits from years of
development and architecture-specific tuning.

The proposed algorithm is closed-form and can be inlined in performance-critical code sections,
while LAPACK routines typically involve function-call overhead. This can further widen the
performance gap in favor of the proposed algorithm in practical scenarios.

\section{Application: Evaluation of the Mohr--Coulomb yield function in mechanics}

The importance of accurate evaluation of eigenvalues of $3 \times 3$ matrices in practical applications
cannot be overstated.
In civil and mechanical engineering the prediction of material failure is based on eigenvalues of
the stress tensor, $\sigma_1, \sigma_2$ and $\sigma_3$, \citep[\S II.2]{Hill1998Plasticity}.
For rigid body dynamics, the analysis
of stability of rotation is based on eigenvalues and eigenvectors of the inertia tensor
\citep[Eq. 5.29]{Goldstein2001Mechanics}. For medical
imaging, diffusion tensor imaging (DTI) relies on eigenvalues of the diffusion tensor to characterize
the anisotropic diffusion of water molecules in biological tissues (indicates e.g., damage to the tissue
due to injury or disease), \citep[Eq. 10]{Basser1994Diffusion}. Another important application, especially
for nonsymmetric matrices, is the analysis of fluid flows and stability of vortices, \citep{Chong1990Flow}.

In this section, we demonstrate the practical impact of numerical stability in eigenvalue computation
by evaluating the Mohr--Coulomb yield function, a widely used model in geotechnical engineering and
soil mechanics to predict the onset of plastic deformation in materials such as soil and rock.
The Mohr--Coulomb yield function is defined as (see \citet[\S 3.3.3.]{Lubliner2008Plasticity}
and \citet[\S 2.3.3.]{Chen2007Plasticity})
\begin{equation}
    f(\bm{\sigma}) = \frac{1}{S_{yc}}\frac{m + 1}{2} \max_{i < j} \bigl( \abs{\sigma_i - \sigma_j} + K (\sigma_i + \sigma_j) \bigr) - 1,
    \label{eq:mohr-coulomb-yield}
\end{equation}
where $\sigma_1, \sigma_2$ and $\sigma_3$ are the principal stresses (eigenvalues of the stress tensor
$\bm{\sigma} \in \mathbb R^{3 \times 3}$), $S_{yc}$ is the uniaxial compressive yield stress,
$m$ is the ratio of the uniaxial tensile yield stress $S_{yt}$ to the uniaxial compressive yield stress,
$m = S_{yc} / S_{yt}$, and $K = (m - 1) / (m + 1)$ is a material parameter related to the internal friction angle.
In this example we take $S_{yc} = \SI{100}{\mega\pascal}$ and $S_{yt} = \SI{10}{\mega\pascal}$ (i.e., $m = 10$).

The material behaves elastically when $f(\bm{\sigma}) < 0$ and yields when $f(\bm{\sigma}) = 0$,
indicating the onset of plastic deformation. Moreover, the gradient of the yield function with respect to the stress tensor,
$\partial f / \partial \bm{\sigma}$, is essential for numerical algorithms in computational plasticity,
such as return mapping algorithms, which are used to update the stress state during plastic deformation.
Accurate evaluation of the yield function is crucial for predicting material failure and designing safe structures.

We evaluate the absolute forward error in computing the Mohr--Coulomb yield function
\begin{equation}
    \abs{\fl(f(\bm{\sigma})) - f(\bm{\sigma})}
\end{equation}
using the eigenvalue Algorithm~\ref{alg:eigvals-stable} with naive invariant computations
(Algorithms~\ref{alg:J2-naive}, and \ref{alg:J3-naive}) and using stable invariant
computations (Algorithms~\ref{alg:J2-stable}, \ref{alg:J3-stable}, and \ref{alg:disc-stable}).

Results for the forward error over a range of stress states are shown in
Fig.~\ref{fig:mohr-coulomb-error}. The figure shows a polar plot in the
so called $\pi$-plane, which is commonly used in soil mechanics.
The $\pi$-plane is defined as the plane orthogonal to the hydrostatic axis in the principal stress space,
which is the axis where all three principal stresses are equal, i.e., $\sigma_1 = \sigma_2 = \sigma_3$.
The radial axis is logarithmic, with the center corresponding to an error of $10^{-16}$.
The angular coordinate represents the Lode angle $\theta$, related to the triple-angle defined earlier as
$\theta = \varphi / 3$. The evaluation is performed for a stress state constructed as
\begin{equation}
    \bm{\sigma} = \mathbf U_\text{symm} \diag(\sigma_1, \sigma_2, \sigma_3) \mathbf U_\text{symm}^\T,
\end{equation}
where $\mathbf U_\text{symm}$ is the orthogonal matrix defined in Eq.~\eqref{eq:transformation-matrices},
and the principal stresses are computed from the Lode angle $\theta$ by
\begin{equation}
    \sigma_k = \frac{2}{3} \sqrt{3 J_2} \cos \left( \theta + \frac{2 \pi k}{3} \right), \quad k \in \{1, 2, 3\},
\end{equation}
where $J_2 = \SI{150}{\mega\pascal\squared}$ is a constant scalar and $\theta \in [0, 2\pi)$ is
the angle in the $\pi$-plane.

\def\centerLog{-16}

\begin{figure}[htbp]
    \centering
    \begin{subfigure}[t]{0.49\linewidth}
        \begin{tikzpicture}
            \begin{polaraxis}[
                    width=7cm,
                    height=7cm,
                    grid=both,
                    tick label style={font=\small},
                    label style={font=\small},
                    ymin=\centerLog,
                    ymax=0,
                    xlabel={Lode angle $\theta$},
                    y coord trafo/.code={
                            \pgfmathparse{#1 - \centerLog}
                        },
                    y coord inv trafo/.code={
                            \pgfmathparse{#1 + \centerLog}
                        },
                    legend style={font=\small, at={(0.5,1.15)}, anchor=south, legend columns=2, /tikz/every even column/.append style={column sep=0.5cm}},
                ]
                \addplot+ [data cs=polarrad, red, only marks,mark=triangle, mark size=2pt] table [x=angle,y=error_naive] {results/mohr-coulomb-error.dat};
                \addlegendentry{abs. fwd. error}
                \addplot [data cs=polarrad, black, mark=none, thick] table [x=angle,y=yield_surface] {results/mohr-coulomb-error.dat};
                \addlegendentry{$\abs{f(\bm{\sigma})}$}
            \end{polaraxis}
        \end{tikzpicture}
        \caption{Naive implementation}
    \end{subfigure}
    \hfill
    \begin{subfigure}[t]{0.49\linewidth}
        \begin{tikzpicture}
            \begin{polaraxis}[
                    width=7cm,
                    height=7cm,
                    grid=both,
                    tick label style={font=\small},
                    label style={font=\small},
                    ymin=\centerLog,
                    ymax=0, 
                    xlabel={Lode angle $\theta$},
                    y coord trafo/.code={
                            \pgfmathparse{#1 - \centerLog}
                        },
                    y coord inv trafo/.code={
                            \pgfmathparse{#1 + \centerLog}
                        },
                    legend style={font=\small, at={(0.5,1.15)}, anchor=south, legend columns=2, /tikz/every even column/.append style={column sep=0.5cm}},
                ]
                \addplot+ [data cs=polarrad, blue, only marks,mark=x, mark size=2pt] table [x=angle,y=error_stable] {results/mohr-coulomb-error.dat};
                \addlegendentry{abs. fwd. error}
                \addplot [data cs=polarrad, black, mark=none, thick] table [x=angle,y=yield_surface] {results/mohr-coulomb-error.dat};
                \addlegendentry{$\abs{f(\bm{\sigma})}$}
            \end{polaraxis}
        \end{tikzpicture}
        \caption{Present algorithm}
    \end{subfigure}
    \caption{Absolute forward error in evaluating the Mohr--Coulomb yield function
        using naive and present eigenvalue algorithms over a range of stress states. The radial axis is
        logarithmic, with the center corresponds to an absolute error of $10^{-16}$. The absolute value
        of the yield function itself is also plotted (black line).}
    \label{fig:mohr-coulomb-error}
\end{figure}

Figure \ref{fig:mohr-coulomb-error} shows that the naive implementation produces large errors, up to
$\epsmach^{1/2} \approx 10^{-8}$ in accordance with the analysis in
Section~\ref{sec:eigvals-stability}. The error is particularly large near Lode angles
$\theta = \{0, 60, 120, 180, 240, 300\}^\circ$, which correspond to stress states with two coalescing eigenvalues,
or equivalently the vanishing discriminant, $\Delta = 0$. In contrast, the present stable algorithm
maintains an absolute error close to machine precision across the entire range of stress states.
Absolute value of the yield function is also included in Fig.~\ref{fig:mohr-coulomb-error} as a
black line. As can be seen in the figure, the yield function approaches zero near another set of
Lode angles. In this case, the relative error of both algorithms would increase, since the stability
analysis from the Lemma~\ref{lem:eigvals-fwd-stability} only provides an absolute error bound.
Development of algorithms for yield surfaces with guaranteed relative error bounds remains a topic
for future research.

\section{Conclusion}

In this work, we have presented a detailed numerical stability analysis of closed-form expressions
for the eigenvalues of $3 \times 3$ real matrices. We have focused on the computation of four key
invariants: the trace $I_1$, the deviatoric invariants $J_2$ and $J_3$, and the discriminant
$\Delta$. For each invariant, we derived forward error bounds and proposed specialized algorithms
designed to be stable, particularly in the challenging cases of coalescing eigenvalues.

Our analysis and numerical benchmarks demonstrate that the proposed algorithm for the invariant
$J_2$ is accurate, satisfying the derived stability bounds even for matrices with ill-conditioned
eigenbases. The algorithms for $J_3$ and the discriminant $\Delta$, however, are stable for matrices
with well-conditioned or orthogonal eigenbases but their accuracy degrades significantly for
matrices with ill-conditioned eigenbases, where they fail to meet the theoretical stability bounds.
The final eigenvalue computation, which relies on these invariants, inherits their stability
characteristics. Consequently, the proposed closed-form solution is numerically stable and accurate
for matrices with a well-conditioned eigenbasis, significantly outperforming naive implementations.
For the most challenging cases involving ill-conditioned eigenbases, the established iterative
library routine LAPACK DGEEV provides more accurate results.

Performance benchmarks show that the proposed closed-form algorithm is approximately ten times
faster than the highly optimized LAPACK implementation for a challenging test case with nearly
double eigenvalue. This highlights the potential of closed-form solutions in performance-critical
applications, especially considering that our implementation is not fully optimized and could be
inlined to avoid function call overhead.

The header-only open-source library \texttt{eig3x3} \citep{Habera2025Eig3x3} provides the C implementations of
the proposed algorithms with a thin Python interface (via CFFI), including eigenvalue routines for
both general and symmetric $3 \times 3$ matrices: \texttt{eig3x3.eigvals} (real, diagonalizable) and
\texttt{eig3x3.eigvalss} (real, symmetric). The repository includes naive and stable variants for
$J_2$, $J_3$, and $\Delta$, together with build scripts and benchmarks to reproduce the
results reported here.

Finally, we demonstrated the practical impact of numerical stability in the context of the
Mohr--Coulomb yield function in mechanics. The proposed algorithm computes the yield function with
errors close to machine precision across all stress states, whereas naive methods introduce
significant errors near coalescing eigenvalues.

Future work should focus on proving the forward stability of invariants $J_3$ and $\Delta$ (and
consequently the eigenvalues), which we currently only observe empirically for well-conditioned
eigenbases. Further performance gains could be achieved by architecture-specific optimizations, such
as vectorization. In conclusion, while iterative methods remain the gold standard for accuracy in
the most ill-conditioned problems, our work provides a robust and efficient closed-form alternative
for a large and practical class of matrices.

\section*{Declarations}

This version of the article has been accepted for publication, after peer review and is subject to
Springer Nature's AM terms of use, but is not the Version of Record and does not reflect post-acceptance
improvements, or any corrections. The Version of Record is available online at: http://dx.doi.org/10.1007/s11075-026-02328-5

\bibliographystyle{unsrtnat}
\bibliography{bibliography}

\begin{thebibliography}{29}
\providecommand{\natexlab}[1]{#1}
\providecommand{\url}[1]{\texttt{#1}}
\expandafter\ifx\csname urlstyle\endcsname\relax
  \providecommand{\doi}[1]{doi: #1}\else
  \providecommand{\doi}{doi: \begingroup \urlstyle{rm}\Url}\fi

\bibitem[Smith(1961)]{Smith1961Eigenvalues}
Oliver~K Smith.
\newblock Eigenvalues of a symmetric 3$\times$ 3 matrix.
\newblock \emph{Communications of the ACM}, 4\penalty0 (4):\penalty0 168, 1961.

\bibitem[Bronshtein et~al.(2015)Bronshtein, Semendyayev, Musiol, and M\"{u}hlig]{Bronshtein2015Handbook}
I.N. Bronshtein, K.A. Semendyayev, Gerhard Musiol, and Heiner M\"{u}hlig.
\newblock \emph{Handbook of Mathematics}.
\newblock Springer, Berlin, Heidelberg, 2015.
\newblock ISBN 9783662462218.
\newblock \doi{10.1007/978-3-662-46221-8}.
\newblock URL \url{http://dx.doi.org/10.1007/978-3-662-46221-8}.

\bibitem[Press et~al.(2007)Press, Teukolsky, Vetterling, and Flannery]{Press2007Recipes}
William~H. Press, Saul~A. Teukolsky, William~T. Vetterling, and Brian~P. Flannery.
\newblock \emph{Numerical Recipes 3rd Edition: The Art of Scientific Computing}.
\newblock Cambridge University Press, USA, 3 edition, 2007.
\newblock ISBN 0521880688.

\bibitem[Trefethen and Bau(1997)]{Trefethen1997Numerical}
Lloyd~N. Trefethen and David Bau.
\newblock \emph{Numerical Linear Algebra}.
\newblock Society for Industrial and Applied Mathematics, Philadelphia, PA, January 1997.
\newblock ISBN 9780898719574.
\newblock \doi{10.1137/1.9780898719574}.
\newblock URL \url{http://dx.doi.org/10.1137/1.9780898719574}.

\bibitem[Higham(2002)]{Higham2002Accuracy}
Nicholas~J. Higham.
\newblock \emph{Accuracy and Stability of Numerical Algorithms}.
\newblock Society for Industrial and Applied Mathematics, Philadelphia, PA, January 2002.
\newblock ISBN 9780898718027.
\newblock \doi{10.1137/1.9780898718027}.
\newblock URL \url{http://dx.doi.org/10.1137/1.9780898718027}.

\bibitem[Blinn(2007)]{Blinn2007Howto}
James~F. Blinn.
\newblock How to solve a cubic equation, part 5: Back to numerics.
\newblock \emph{IEEE Computer Graphics and Applications}, 27\penalty0 (3):\penalty0 78–89, May 2007.
\newblock ISSN 1558-1756.
\newblock \doi{10.1109/mcg.2007.60}.
\newblock URL \url{http://dx.doi.org/10.1109/MCG.2007.60}.

\bibitem[La~Porte(1973)]{LaPorte1973Formulation}
M~La~Porte.
\newblock Une formulation num{\'e}riquement stable donnant les racines r{\'e}elles de l’{\'e}quation du 3’eme degr{\'e}.
\newblock \emph{IFP Report}, 21516, 1973.

\bibitem[Scherzinger and Dohrmann(2008)]{Scherzinger2008Robust}
W.M. Scherzinger and C.R. Dohrmann.
\newblock A robust algorithm for finding the eigenvalues and eigenvectors of 3×3 symmetric matrices.
\newblock \emph{Computer Methods in Applied Mechanics and Engineering}, 197\penalty0 (45–48):\penalty0 4007–4015, August 2008.
\newblock ISSN 0045-7825.
\newblock \doi{10.1016/j.cma.2008.03.031}.
\newblock URL \url{http://dx.doi.org/10.1016/j.cma.2008.03.031}.

\bibitem[Habera and Zilian(2021)]{Habera2021Symbolic}
Michal Habera and Andreas Zilian.
\newblock Symbolic spectral decomposition of 3x3 matrices, 2021.
\newblock URL \url{https://arxiv.org/abs/2111.02117}.

\bibitem[Parlett(2002)]{Parlett2002Discriminant}
Beresford~N. Parlett.
\newblock The (matrix) discriminant as a determinant.
\newblock \emph{Linear Algebra and its Applications}, 355\penalty0 (1–3):\penalty0 85–101, November 2002.
\newblock ISSN 0024-3795.
\newblock \doi{10.1016/s0024-3795(02)00335-x}.
\newblock URL \url{http://dx.doi.org/10.1016/S0024-3795(02)00335-X}.

\bibitem[Harari and Albocher(2022)]{Harari2022Computation}
Isaac Harari and Uri Albocher.
\newblock Computation of eigenvalues of a real, symmetric 3 × 3 matrix with particular reference to the pernicious case of two nearly equal eigenvalues.
\newblock \emph{International Journal for Numerical Methods in Engineering}, 124\penalty0 (5):\penalty0 1089–1110, November 2022.
\newblock ISSN 1097-0207.
\newblock \doi{10.1002/nme.7153}.
\newblock URL \url{http://dx.doi.org/10.1002/nme.7153}.

\bibitem[Harari and Albocher(2023)]{Harari2023Using}
Isaac Harari and Uri Albocher.
\newblock Using the discriminant in a numerically stable symmetric 3 × 3 direct eigenvalue solver.
\newblock \emph{International Journal for Numerical Methods in Engineering}, 124\penalty0 (20):\penalty0 4473–4489, July 2023.
\newblock ISSN 1097-0207.
\newblock \doi{10.1002/nme.7311}.
\newblock URL \url{http://dx.doi.org/10.1002/nme.7311}.

\bibitem[Anderson et~al.(1999)Anderson, Bai, Bischof, Blackford, Demmel, Dongarra, Du~Croz, Greenbaum, Hammarling, McKenney, and Sorensen]{Anderson1999Lapack}
E.~Anderson, Z.~Bai, C.~Bischof, S.~Blackford, J.~Demmel, J.~Dongarra, J.~Du~Croz, A.~Greenbaum, S.~Hammarling, A.~McKenney, and D.~Sorensen.
\newblock \emph{{LAPACK} Users' Guide}.
\newblock Society for Industrial and Applied Mathematics, Philadelphia, PA, third edition, 1999.
\newblock ISBN 0-89871-447-8 (paperback).

\bibitem[Habera and Zilian(2025)]{Habera2025Eig3x3}
Michal Habera and Andreas Zilian.
\newblock Eigenvalues of 3x3 matrices, 2025.
\newblock URL \url{https://github.com/michalhabera/eig3x3}.

\bibitem[Rigo and Fijalkowski(2025)]{Rigo2025CFFI}
Armin Rigo and Maciej Fijalkowski.
\newblock Cffi documentation, 2025.
\newblock URL \url{https://cffi.readthedocs.io/en/stable/}.
\newblock C Foreign Function Interface for Python. Version 2.0.0.

\bibitem[Harris et~al.(2020)Harris, Millman, van~der Walt, Gommers, Virtanen, Cournapeau, Wieser, Taylor, Berg, Smith, Kern, Picus, Hoyer, van Kerkwijk, Brett, Haldane, del R{\'{i}}o, Wiebe, Peterson, G{\'{e}}rard-Marchant, Sheppard, Reddy, Weckesser, Abbasi, Gohlke, and Oliphant]{Harris2020Array}
Charles~R. Harris, K.~Jarrod Millman, St{\'{e}}fan~J. van~der Walt, Ralf Gommers, Pauli Virtanen, David Cournapeau, Eric Wieser, Julian Taylor, Sebastian Berg, Nathaniel~J. Smith, Robert Kern, Matti Picus, Stephan Hoyer, Marten~H. van Kerkwijk, Matthew Brett, Allan Haldane, Jaime~Fern{\'{a}}ndez del R{\'{i}}o, Mark Wiebe, Pearu Peterson, Pierre G{\'{e}}rard-Marchant, Kevin Sheppard, Tyler Reddy, Warren Weckesser, Hameer Abbasi, Christoph Gohlke, and Travis~E. Oliphant.
\newblock Array programming with {NumPy}.
\newblock \emph{Nature}, 585\penalty0 (7825):\penalty0 357--362, September 2020.
\newblock \doi{10.1038/s41586-020-2649-2}.
\newblock URL \url{https://doi.org/10.1038/s41586-020-2649-2}.

\bibitem[Fredrik~Johansson(2023)]{Mpmath2023Python}
et~al. Fredrik~Johansson.
\newblock \emph{mpmath: a Python library for arbitrary-precision floating-point arithmetic (version 1.3.0)}, 2023.
\newblock http://mpmath.org/.

\bibitem[{NumPy}(2025{\natexlab{a}})]{Numpy2025Trace}
{NumPy}.
\newblock numpy.trace, 2025{\natexlab{a}}.
\newblock URL \url{https://numpy.org/doc/2.3/reference/generated/numpy.trace.html}.
\newblock NumPy v2.3 Manual.

\bibitem[{NumPy}(2025{\natexlab{b}})]{Numpy2025Matmul}
{NumPy}.
\newblock numpy.matmul, 2025{\natexlab{b}}.
\newblock URL \url{https://numpy.org/doc/2.3/reference/generated/numpy.matmul.html}.
\newblock NumPy v2.3 Manual.

\bibitem[{NumPy}(2025{\natexlab{c}})]{Numpy2025Det}
{NumPy}.
\newblock numpy.linalg.det, 2025{\natexlab{c}}.
\newblock URL \url{https://numpy.org/doc/2.3/reference/generated/numpy.linalg.det.html}.
\newblock NumPy v2.3 Manual.

\bibitem[Bauer and Fike(1960)]{Bauer1960Norms}
F.~L. Bauer and C.~T. Fike.
\newblock Norms and exclusion theorems.
\newblock \emph{Numerische Mathematik}, 2\penalty0 (1):\penalty0 137–141, December 1960.
\newblock ISSN 0945-3245.
\newblock \doi{10.1007/bf01386217}.
\newblock URL \url{http://dx.doi.org/10.1007/BF01386217}.

\bibitem[Golub and Van~Loan(2013)]{Golub2013Matrix}
Gene~H. Golub and Charles~F. Van~Loan.
\newblock \emph{Matrix Computations - 4th Edition}.
\newblock Johns Hopkins University Press, Philadelphia, PA, 2013.
\newblock \doi{10.1137/1.9781421407944}.
\newblock URL \url{https://epubs.siam.org/doi/abs/10.1137/1.9781421407944}.

\bibitem[{NumPy}(2025{\natexlab{d}})]{Numpy2025Eigvals}
{NumPy}.
\newblock numpy.linalg.eigvals, 2025{\natexlab{d}}.
\newblock URL \url{https://numpy.org/doc/2.3/reference/generated/numpy.linalg.eigvals.html}.
\newblock NumPy v2.3 Manual.

\bibitem[Hill(1998)]{Hill1998Plasticity}
R~Hill.
\newblock \emph{The Mathematical Theory Of Plasticity}.
\newblock Oxford University Press, Oxford, August 1998.
\newblock ISBN 9781383020625.
\newblock \doi{10.1093/oso/9780198503675.001.0001}.
\newblock URL \url{http://dx.doi.org/10.1093/oso/9780198503675.001.0001}.

\bibitem[Goldstein et~al.(2001)Goldstein, Poole, and Safko]{Goldstein2001Mechanics}
Herbert Goldstein, Charles~P Poole, and John~L Safko.
\newblock \emph{Classical Mechanics}.
\newblock Pearson, Upper Saddle River, NJ, 3 edition, June 2001.

\bibitem[Basser et~al.(1994)Basser, Mattiello, and LeBihan]{Basser1994Diffusion}
P.J. Basser, J.~Mattiello, and D.~LeBihan.
\newblock Mr diffusion tensor spectroscopy and imaging.
\newblock \emph{Biophysical Journal}, 66\penalty0 (1):\penalty0 259–267, January 1994.
\newblock ISSN 0006-3495.
\newblock \doi{10.1016/s0006-3495(94)80775-1}.
\newblock URL \url{http://dx.doi.org/10.1016/s0006-3495(94)80775-1}.

\bibitem[Chong et~al.(1990)Chong, Perry, and Cantwell]{Chong1990Flow}
M.~S. Chong, A.~E. Perry, and B.~J. Cantwell.
\newblock A general classification of three‐dimensional flow fields.
\newblock \emph{Physics of Fluids A: Fluid Dynamics}, 2\penalty0 (5):\penalty0 765--777, 05 1990.
\newblock ISSN 0899-8213.
\newblock \doi{10.1063/1.857730}.
\newblock URL \url{https://doi.org/10.1063/1.857730}.

\bibitem[Lubliner(2008)]{Lubliner2008Plasticity}
J.~Lubliner.
\newblock \emph{Plasticity Theory}.
\newblock Dover books on engineering. Dover Publications, Mineola, NY, 2008.
\newblock ISBN 9780486462905.

\bibitem[Chen et~al.(2007)Chen, Han, and Han]{Chen2007Plasticity}
W.F. Chen, D.J. Han, and D.J. Han.
\newblock \emph{Plasticity for Structural Engineers}.
\newblock J. Ross Publishing Classics. J. Ross Pub., Plantation, FL, 2007.
\newblock ISBN 9781932159752.

\end{thebibliography}

\end{document}